\documentclass[11pt,a4paper]{article}

\usepackage{amsmath,amssymb,amsthm,mathtools}
\usepackage{geometry}
\usepackage{placeins}
\usepackage{booktabs,array,longtable}
\usepackage{microtype}
\usepackage{xcolor}
\usepackage{tikz}
\usepackage[numbers,sort&compress]{natbib}
\usepackage[colorlinks=true,linkcolor=blue!45!black,
  citecolor=green!35!black,urlcolor=blue!50!black]{hyperref}

\newtheorem{theorem}{Theorem}[section]
\newtheorem{proposition}[theorem]{Proposition}
\newtheorem{lemma}[theorem]{Lemma}

\theoremstyle{remark}

\title{$6$ powers of $2$ in the Linnik--Goldbach problem}
\author{Huixi Li
  \thanks{\raggedright School of Mathematical Sciences and LPMC,
  Nankai University, Tianjin, China. Email
  \mbox{lihuixi@nankai.edu.cn}.}
\and
Zhixin Liu
  \thanks{\raggedright Corresponding author. School of Mathematical Sciences, Tianjin University,
  Tianjin, China. Email \mbox{zhixinliu@tju.edu.cn}.}
}
\date{}

\begin{document}
\maketitle

\begin{abstract}
Linnik's approximation to Goldbach's conjecture concerns the representation
of sufficiently large even integers as the sum of two primes and a bounded
number of powers of two. 
We prove that $6$ powers of $2$ suffice unconditionally and
that $5$ suffice under the GRH, improving previous works. 
\end{abstract}



\section{Introduction}\label{sec:introduction}

Linnik's approximation to Goldbach's conjecture asks for a fixed integer
$K$ such that every sufficiently large even integer has a representation
\begin{equation}\label{eq:linnik}
 N=p_1+p_2+\sum_{j=1}^k2^{\nu_j},\qquad 0\le k\le K,
\end{equation}
where $p_1,p_2$ are primes and the exponents are positive integers. 
By comparison, Chen~\cite{Chen} proved in 1973 that every sufficiently
large even integer is the sum of a prime and an integer with at most
two prime factors.
Linnik~\cite{Linnik1,Linnik2} proved the existence of $K$ under the
generalised Riemann hypothesis (GRH) in 1951 and unconditionally in
1953. Gallagher~\cite{Gallagher} simplified the argument in 1975.
Table~\ref{tab:history} includes the subsequent explicit bounds.

In 2026, Johnston and Trudgian~\cite{JT} proved that $7$ powers of $2$ suffice unconditionally and $6$ powers 
under GRH.
\begin{table}[htbp]
\centering
\small
\caption{Selected admissible bounds in \eqref{eq:linnik} for every
sufficiently large even integer.}
\begin{tabular*}{\linewidth}{@{\extracolsep{\fill}}llrr@{}}
\toprule
Year & Work & Unconditional $K$ & GRH $K$ \\
\midrule
1951 & Linnik~\cite{Linnik1} & -- & unspecified \\
1953 & Linnik~\cite{Linnik2} & unspecified & -- \\
1975 & Gallagher~\cite{Gallagher} & unspecified & -- \\
1998 & Liu--Liu--Wang (I)~\cite{LLWI} & -- & 770 \\
1998 & Liu--Liu--Wang (II)~\cite{LLW} & 54000 & -- \\
1999 & Liu--Liu--Wang~\cite{LLWIII} & -- & 200 \\
1999 & Wang~\cite{Wang} & 2250 & 160 \\
2000 & Li~\cite{Li1} & 25000 & -- \\
2001 & Li~\cite{Li2} & 1906 & -- \\
2002 & Heath-Brown--Puchta~\cite{HP} & 13 & 7 \\
2003 & Pintz--Ruzsa~\cite{PRI} & -- & 7 \\
2011 & Liu--L\"u~\cite{LiuLu} & 12 & -- \\
2020 & Pintz--Ruzsa (II)~\cite{PRII} & 8 & -- \\
2026 & Johnston--Trudgian~\cite{JT} & 7 & 6 \\
2026 & Present paper & 6 & 5 \\
\bottomrule
\end{tabular*}
\label{tab:history}
\end{table}

We reduce these bounds to $6$ unconditionally and $5$ under GRH.

For a sufficiently small fixed $0<\kappa<1/2$, put
\begin{equation}\label{eq:NL}
 N_1=N^{1-\kappa},\qquad
 L=\left\lfloor\log_2N-\sqrt{\log_2N}\right\rfloor.
\end{equation}
For $K\in\{5,6\}$, define
\begin{equation}\label{eq:R}
 r''_K(N)=
 \sum_{\substack{N=p_1+p_2+2^{\nu_1}+\cdots+2^{\nu_K}\\
 N_1<p_1,p_2\le N,\quad 1\le\nu_1,\ldots,\nu_K\le L}}
 (\log p_1)(\log p_2).
\end{equation}

\begin{theorem}\label{thm:main}
Every sufficiently large even integer is the sum of two primes and at most $6$ powers of $2$. In particular, there are constants $0<\kappa<1/2$ and $N_0$ such that every even integer
$N\ge N_0$ satisfies
\begin{equation}\label{eq:maincount}
 r''_6(N)\ge\frac1{1000}NL^6.
\end{equation}
\end{theorem}

\begin{theorem}\label{thm:GRHfive}
Assume GRH. Every sufficiently large even integer is the sum of two primes and at most $5$ powers of $2$. In particular, there are constants $0<\kappa<1/2$ and $N_0$ such that every
even integer $N\ge N_0$ satisfies
\begin{equation}\label{eq:fivecount}
 r''_5(N)\ge\frac1{1000}NL^5.
\end{equation}
\end{theorem}

Our main innovation is to combine signed representation functions with joint moment estimates that retain the complex argument of the sum over powers of $2$. The correction terms produce cancellation in the second moment while leaving a sufficiently large first moment. Earlier estimates lose part of this cancellation by taking absolute values or estimating the relevant terms separately. Retaining this information gives the saving needed to reduce $7$ powers to $6$ unconditionally, and $6$ powers to $5$ under GRH.

In a forthcoming paper we will prove every sufficiently large even integer is the sum of two primes and at most $5$ powers of $2$.

To describe the proof, write $e(t)=\exp(2\pi it)$ and put
\begin{equation}\label{eq:SG}
 S(\alpha)=\sum_{N_1<p\le N}(\log p)\cdot e(p\alpha),\qquad
 G(\alpha)=\sum_{\nu=1}^Le(2^\nu\alpha).
\end{equation}
We also write
\begin{equation}\label{eq:zL}
 z_L(\alpha)=G(\alpha)/L.
\end{equation}
For fixed $a,b\ge0$ with $a+b<1$, we use
\[
 f(n)=r'_3(n)-aL\cdot r'_2(n)-bL^2\cdot r'_1(n),\qquad
 Q(z)=z^2-az-b,
\]
where $r'_j(n)$ is the one-prime representation function defined in
\eqref{eq:rprime}. Define $g(n)=\max\{f(n),0\}$. Then
\[
 0\le g(n)\le r'_3(n),\qquad
 \sum_{\substack{1\le n<N\\n\text{ odd}}}g(n)
 \ge(1-a-b+o(1))NL^3.
\]

Logarithms without a subscript are natural, and $\varphi$ and $\mu$
denote Euler's totient function and the M\"obius function.
For $1\le j\le3$, put
\begin{equation}\label{eq:rprime}
 r'_j(n)=\sum_{\substack{n=p+2^{\nu_1}+\cdots+2^{\nu_j}\\
 N_1<p\le N,\quad 1\le\nu_1,\ldots,\nu_j\le L}}\log p.
\end{equation}
Equivalently,
\begin{equation*}
 r'_j(n)=\int_0^1S(\alpha)\cdot G(\alpha)^j e(-n\alpha)\,\mathrm d\alpha.
\end{equation*}
A prime on a summation sign $\sum'$ restricts
the index to positive odd squarefree integers. All asymptotic estimates
are taken as $N\to\infty$, with the stated parameters fixed.
A bar denotes complex conjugation, and $v^T$ denotes the transpose
of a column vector $v$.
The principal notation is summarized in Table~\ref{tab:notation}.

\FloatBarrier
\begingroup
\small
\renewcommand{\arraystretch}{1.13}
\begin{longtable}{@{}>{\raggedright\arraybackslash}p{.24\linewidth}>{\raggedright\arraybackslash}p{.71\linewidth}@{}}
\caption{Principal notation}
\label{tab:notation}\\
\toprule
Notation & Meaning\\
\midrule
\endfirsthead
\multicolumn{2}{@{}l}{\textit{Table \thetable\ continued.}}\\
\toprule
Notation & Meaning\\
\midrule
\endhead
\midrule
\multicolumn{2}{r@{}}{\textit{Continued on the next page.}}\\
\endfoot
\bottomrule
\endlastfoot
$N,N_1,L,\kappa$ & The even integer, prime cutoff, exponent cutoff
and fixed parameter in \eqref{eq:NL}.\\
$r''_K(N),r'_j(n)$ & Weighted ordered counts with two primes and
$K$ powers, or one prime and $j$ powers; see \eqref{eq:R}, \eqref{eq:rprime}.\\
$e(t),S(\alpha),G(\alpha)$ & $e(t)=\exp(2\pi it)$ and the prime
and dyadic sums in \eqref{eq:SG}.\\
$z_L(\alpha)$ & The normalized sum $G(\alpha)/L$ in \eqref{eq:zL}.\\
$f(n),g(n),Q(z),F(z)$ & The signed function, its positive part and
correcting polynomials in \eqref{eq:positivepart}--\eqref{eq:parameters}.\\
$U_N$ & The set $\{1\le n<N:n\text{ odd}\}$ in Section~\ref{sec:positive}.\\
\addlinespace[4pt]
$r_{i,j}(h)$ & Ordered counts for $h=\sum_{t=1}^i2^{\nu_t}
-\sum_{t=1}^j2^{\mu_t}$, with exponents in $[1,L]$; see Section~\ref{sec:singular}.\\
$\mathfrak S(h),C_0$ & The singular series and constant product in
\eqref{eq:sigma}.\\
$\epsilon(d),R_0$ & The order of $2$ modulo odd $d$, with
$\epsilon(1)=1$, and the Romanov series in \eqref{eq:Romanov}.\\
$z_q(a)$ & The average of $e(a2^t/q)$ over one period; see Section~\ref{sec:singular}.\\
$A(q),A,A_*$ & The finite Gram matrix, its full sum, and its part
with $q=3,5,7,15,31$; see Section~\ref{sec:singular}.\\
$v,w$ & The vectors $(-b,-a,1)^T$ and $(-c,1,0)^T$.\\
\addlinespace[4pt]
$\mathfrak M,\mathfrak m,P$ & The unconditional arcs and arc parameter
in \eqref{eq:Prange}--\eqref{eq:arcs}.\\
$R(h),R_1(h)$ & The integrals of $|S(\alpha)|^2e(-h\alpha)$ over
$[0,1]$ and $\mathfrak M$ in \eqref{eq:prime-integrals}.\\
$\mathfrak M_{\mathrm G},\mathfrak m_{\mathrm G},R_{1,\mathrm G}(h)$ &
The GRH arcs and major-arc integral in Lemma~\ref{lem:grhbridge}.\\
$C_1,\varepsilon$ & The sieve constant $C_1=6.7814$ and the fixed
error parameter in Section~\ref{sec:arcs}.\\
\addlinespace[4pt]
$w_{\xi,\eta},V_{\xi,\eta},\Phi_{\xi,\eta}$ & The trigonometric sum,
quadratic average and directional exponent in Section~\ref{sec:large}.\\
$I(\lambda),\rho$ & The entropy bound in \eqref{eq:entropy} and a
decay exponent $0<\rho<I(\lambda)$.\\
$r_\lambda,\beta_\lambda,\mathcal D_\lambda$ & The level
$\lambda^2/2$, radial bound and region in
\eqref{eq:region-parameters}--\eqref{eq:region}.\\
$\mathcal E_\lambda(L),\mathcal E_6,\mathcal E_5$ & The exceptional
set in Proposition~\ref{prop:large} and its cases $\lambda=5/6,25/32$.\\
$\gamma_6,\gamma_5$ & The exponents $81/(200\log2)$ and
$347/(1000\log2)$ in Sections~\ref{sec:completion}--\ref{sec:five}.\\
\addlinespace[4pt]
$f_j(n),g_j(n),F_j(z)$ & The signed functions, positive parts and
polynomials for five powers ($j=2,3$); see Section~\ref{sec:five}.\\
$\tau,D,H$ & The joint-moment weight, diagonal coefficient and
polynomial maximum in Section~\ref{sec:five}.\\
$\mathcal I,\mathcal I_E$ & The joint second moment in
\eqref{eq:jointmoment} and its restriction to $E\subseteq[0,1]$.\\
\end{longtable}
\endgroup

\FloatBarrier

In Section 2, a reflection argument reduces the representation problem to estimates for second moments. Section 3 establishes the mixed singular-series averages needed for these estimates. Section 4 treats the major and minor arcs, while Section 5 estimates the sum over powers of $2$ outside a small exceptional set. We combine these results in Section 6 to prove the unconditional theorem, and in Section 7 to prove the theorem under GRH.

\section{Signed representation functions}\label{sec:positive}

Fix $a,b\ge0$ with $a+b<1$. With $r'_j$ as in \eqref{eq:rprime}, recall that
\begin{equation}\label{eq:positivepart}
 f(n)=r'_3(n)-aL\cdot r'_2(n)-bL^2\cdot r'_1(n),
 \qquad g(n)=\max\{f(n),0\}.
\end{equation}
Then $0\le g(n)\le r'_3(n)$. Put
\begin{equation}\label{eq:parameters}
 Q(z)=z^2-az-b,\qquad
 F(z)=L^2\cdot z\cdot Q(z/L).
\end{equation}
Thus
\begin{equation*}
 f(n)=\int_0^1S(\alpha)\cdot F(G(\alpha))e(-n\alpha)\,\mathrm d\alpha.
\end{equation*}
Let $U_N=\{1\le n<N:n\text{ odd}\}$, so $|U_N|=N/2$.
Since $3\cdot2^L=o(N)$, the prime number theorem gives
$\sum_{N_1<p<N-h}\log p=(1+o(1))N$ uniformly for
$0\le h\le3\cdot2^L$. Consequently,
\begin{equation}\label{eq:firstmoment}
\begin{aligned}
 \sum_{n\in U_N}r'_j(n)
 &=\sum_{1\le\nu_1,\ldots,\nu_j\le L}
   \sum_{N_1<p<N-\sum_{t=1}^j2^{\nu_t}}\log p\\
 &=(1+o(1))NL^j\qquad(1\le j\le3),\\
 \sum_{n\in U_N}g(n)&\ge\sum_{n\in U_N}f(n)
 =\left(1-a-b+o(1)\right)NL^3.
\end{aligned}
\end{equation}
Also, Parseval gives
\begin{equation}\label{eq:momenttarget}
 \sum_{n\in U_N}g(n)^2\le\sum_{n\in\mathbb Z}f(n)^2
 =\int_0^1|S(\alpha)\cdot F(G(\alpha))|^2\,\mathrm d\alpha.
\end{equation}

\begin{lemma}\label{lem:positive}
For every sufficiently large even $N$,
\begin{equation}\label{eq:reflectionbound}
 r''_6(N)\ge4(1-a-b)^2NL^6-
 \int_0^1|S(\alpha)\cdot F(G(\alpha))|^2\,\mathrm d\alpha+o(NL^6).
\end{equation}
\end{lemma}
\begin{proof}
The reflection $n\mapsto N-n$ permutes $U_N$. Cauchy's inequality gives
\begin{align*}
 r''_6(N)&\ge\sum_{n\in U_N}g(n)g(N-n)\\
 &=\frac12\sum_{n\in U_N}(g(n)+g(N-n))^2-\sum_{n\in U_N}g(n)^2\\
 &\ge\frac4N\left(\sum_{n\in U_N}g(n)\right)^2-\sum_{n\in U_N}g(n)^2\\
 &\ge4(1-a-b)^2NL^6
 -\int_0^1|S(\alpha)\cdot F(G(\alpha))|^2\,\mathrm d\alpha+o(NL^6),
\end{align*}
by \eqref{eq:firstmoment} and \eqref{eq:momenttarget}.
\end{proof}

\section{Mixed singular-series averages}\label{sec:singular}

For $1\le i,j\le3$, let $r_{i,j}(h)$ count the ordered tuples satisfying
$h=\sum_{t=1}^i2^{\nu_t}-\sum_{t=1}^j2^{\mu_t}$,
with all exponents in $[1,L]$. Thus $h$ is even and $|h|\le3\cdot2^L=o(N)$.
For nonzero even $h$, define
\begin{equation}\label{eq:sigma}
 C_0=\prod_{p>2}\left(1-\frac1{(p-1)^2}\right),\qquad
 \mathfrak{S}(h)=2C_0\prod_{\substack{p\mid h\\p>2}}\frac{p-1}{p-2}.
\end{equation}
For odd $d>1$, let $\epsilon(d)$ denote the order of $2$ modulo $d$,
and set $\epsilon(1)=1$. Empty products equal $1$. Then
\begin{equation}\label{eq:Romanov}
 \frac{\mathfrak{S}(h)}{2C_0}
 =\sideset{}{^\prime}\sum_{d\mid h}\frac1{\prod_{p\mid d}(p-2)},
 \qquad
 R_0=\sideset{}{^\prime}\sum_d
 \frac1{\epsilon(d)\prod_{p\mid d}(p-2)}<\infty.
\end{equation}
For convergence, see Pintz and Ruzsa~\cite[Lemma~1]{PRII}.
Platt and Trudgian~\cite[p.~55]{PT} proved that
$C_0R_0<1.2784421041$; we use
\begin{equation}\label{eq:Romanovproduct}
 C_0R_0<\frac{1279}{1000}.
\end{equation}

\begin{lemma}\label{lem:tail}
For fixed $D\ge1$ and $1\le i,j\le3$,
\begin{equation}\label{eq:uniformtail}
 \frac1{L^{i+j}}\sum_{h\ne0}r_{i,j}(h)
 \sideset{}{^\prime}\sum_{\substack{d>D\\d\mid h}}
 \frac1{\prod_{p\mid d}(p-2)}
 \le\sideset{}{^\prime}\sum_{d>D}
 \frac1{\epsilon(d)\prod_{p\mid d}(p-2)}
 +O\left(\frac{\log L}{L}\right).
\end{equation}
If $q$ is odd and tends to infinity with $N$, then 
\begin{equation}\label{eq:badclass}
 \sum_{\substack{h\ne0\\q\mid h}}r_{i,j}(h)\mathfrak{S}(h)=o(L^{i+j}).
\end{equation}
\end{lemma}
\begin{proof}
Fix all but one exponent and, changing the sign if necessary, write
$h=m-2^\nu$, where $|m|\ll2^L$. Put
$\mathcal P_m=\prod_{1\le\nu\le L,\,m\ne2^\nu}|m-2^\nu|$.
Then $\log\mathcal P_m=O(L^2)$. The congruence $2^\nu\equiv m\pmod d$
has at most $L/\epsilon(d)+1$ solutions, then
\begin{align*}
 &\sum_{\substack{1\le\nu\le L\\m\ne2^\nu}}
 \sideset{}{^\prime}\sum_{\substack{d>D\\d\mid m-2^\nu}}
 \frac1{\prod_{p\mid d}(p-2)}
 \le L\sideset{}{^\prime}\sum_{d>D}
 \frac1{\epsilon(d)\prod_{p\mid d}(p-2)}
 +\prod_{\substack{p\mid\mathcal P_m\\p>2}}\left(1+\frac1{p-2}\right).
\end{align*}
By Mertens' estimate,
\begin{align*}
 \prod_{2<p\le L^3}\left(1+\frac1{p-2}\right)&\ll\log L,\\
 \sum_{\substack{p\mid\mathcal P_m\\p>L^3}}\frac1{p-2}
 &\ll\frac1{L^3\log L}\sum_{p\mid\mathcal P_m}\log p
 \le\frac{\log\mathcal P_m}{L^3\log L}
 \ll\frac1{L\log L},\\
 \prod_{\substack{p\mid\mathcal P_m\\p>2}}\left(1+\frac1{p-2}\right)
 &\ll(\log L)\exp\left(O\left(\frac1{L\log L}\right)\right)
 \ll\log L.
\end{align*}
Summing over the remaining exponents proves \eqref{eq:uniformtail}.
With $q\mid h$, the period becomes $\epsilon(\operatorname{lcm}(d,q))$.
Splitting at $D$ therefore gives
\begin{equation*}
 \frac1{2C_0L^{i+j}}\sum_{\substack{h\ne0\\q\mid h}}r_{i,j}(h)\mathfrak{S}(h)
 \le\frac1{\epsilon(q)}\sideset{}{^\prime}\sum_{d\le D}
 \frac1{\prod_{p\mid d}(p-2)}
 +\sideset{}{^\prime}\sum_{d>D}\frac1{\epsilon(d)\prod_{p\mid d}(p-2)}
 +O\left(\frac{\log L}{L}\right).
\end{equation*}
Now $q\mid2^{\epsilon(q)}-1$ gives
\begin{equation*}
 \epsilon(q)\ge\log_2(q+1)\longrightarrow\infty.
\end{equation*}
First let $N\to\infty$ and then $D\to\infty$ in the preceding bound,
using \eqref{eq:Romanov}.
\end{proof}

For every integer $m$, uniqueness of binary expansion gives
\begin{equation*}
 \#\{(\nu,\mu)\in\{1,\ldots,L\}^2:2^\nu+2^\mu=m\}\le2.
\end{equation*}
Fixing all but two exponents on one side, we obtain
\begin{equation}\label{eq:zerocount}
 r_{1,1}(0)=L,\qquad r_{i,j}(0)\le2L^{i+j-2}\quad(i+j\ge3).
\end{equation}
For odd $d$, put
\begin{equation*}
 z_d(a)=\frac1{\epsilon(d)}\sum_{t=1}^{\epsilon(d)}e(a2^t/d),\qquad
 P_d(i,j)=\frac1d\sum_{a\bmod d}z_d(a)^i\overline{z_d(a)^j}.
\end{equation*}
For fixed $d$, periodicity and \eqref{eq:zerocount} give
\begin{equation*}
 \frac1{L^{i+j}}\sum_{\substack{h\ne0\\d\mid h}}r_{i,j}(h)
 =P_d(i,j)+O_d(L^{-1}).
\end{equation*}
Hence Lemma~\ref{lem:tail} gives
\begin{equation}\label{eq:Blimit}
 B_{ij}:=\lim_{N\to\infty}\frac1{2L^{i+j}}
 \sum_{h\ne0}r_{i,j}(h)\mathfrak{S}(h)
 =C_0\sideset{}{^\prime}\sum_d
 \frac{P_d(i,j)}{\prod_{p\mid d}(p-2)}<\infty.
\end{equation}
The diagonal averages are those of Khalfalah and Pintz~\cite{KP},
recorded by Johnston and Trudgian~\cite[Section~3.1]{JT}.
For the signed functions we also need the following mixed form.

\begin{proposition}\label{prop:Gram}
For $1\le i,j\le3$, $B_{ij}=1+A_{ij}$, where
\begin{equation}\label{eq:Agram}
 A_{ij}=\sideset{}{^\prime}\sum_{q>1}\frac1{\varphi(q)^2}
 \sum_{\substack{a\bmod q\\(a,q)=1}}z_q(a)^i\overline{z_q(a)^j}.
\end{equation}
The series converge absolutely. The matrix $A=(A_{ij})_{1\le i,j\le3}$
and every tail obtained by restricting $q$ are real symmetric and
positive semidefinite. Moreover, $A_{11}=C_0R_0-1$.
\end{proposition}
\begin{proof}
Reduction of $a/d$ to lowest terms and the Euler product give
\begin{equation*}
 P_d(i,j)=\frac1d\sum_{q\mid d}\sum_{\substack{a\bmod q\\(a,q)=1}}
 z_q(a)^i\overline{z_q(a)^j},\qquad
 C_0\sideset{}{^\prime}\sum_{\substack{d:q\mid d}}
 \frac1{d\prod_{p\mid d}(p-2)}=\frac1{\varphi(q)^2}.
\end{equation*}
For $i=j$, Tonelli's theorem and \eqref{eq:Blimit} give
\begin{equation*}
 \sideset{}{^\prime}\sum_q\frac1{\varphi(q)^2}
 \sum_{\substack{a\bmod q\\(a,q)=1}}|z_q(a)|^{2i}=B_{ii}<\infty.
\end{equation*}
Cauchy's inequality gives, for all $i,j$,
\begin{equation*}
 \sideset{}{^\prime}\sum_q\frac1{\varphi(q)^2}
 \sum_{\substack{a\bmod q\\(a,q)=1}}|z_q(a)|^{i+j}
 \le\sqrt{B_{ii}B_{jj}}<\infty.
\end{equation*}
Thus the sums in \eqref{eq:Blimit} may be rearranged, and
$B_{ij}=1+A_{ij}$. Since $z_q(-a)=\overline{z_q(a)}$,
\begin{gather*}
 A_{ij}=\overline{A_{ij}}=A_{ji},\\
 v^TAv=\sideset{}{^\prime}\sum_{q>1}\frac1{\varphi(q)^2}
 \sum_{\substack{a\bmod q\\(a,q)=1}}
 \left|\sum_{i=1}^3v_i z_q(a)^i\right|^2\ge0
 \qquad(v\in\mathbb R^3).
\end{gather*}
The same identity applies to every subseries. Finally,
\begin{equation*}
 P_d(1,1)=\frac1{\epsilon(d)},\qquad
 B_{11}=C_0\sideset{}{^\prime}\sum_d
 \frac1{\epsilon(d)\prod_{p\mid d}(p-2)}=C_0R_0.
\end{equation*}
\end{proof}

For each odd squarefree $q>1$, let $A(q)=(A_{ij}(q))_{1\le i,j\le3}$,
where
\begin{equation*}
 A_{ij}(q)=\frac1{\varphi(q)^2}
 \sum_{\substack{a\bmod q\\(a,q)=1}}z_q(a)^i\overline{z_q(a)^j},
 \qquad A=\sum_q A(q).
\end{equation*}
Put
\[
 A_*=A(3)+A(5)+A(7)+A(15)+A(31).
\]
We choose these $5$ moduli because they are precisely the odd integers greater than $1$ for which the multiplicative order of $2$ is at most $5$, allowing their contributions to be computed explicitly.
The cited diagonal bounds and \eqref{eq:Romanovproduct} give
\begin{equation}\label{eq:U}
 A_{11}<0.279,\qquad A_{22}<0.0555,\qquad A_{33}<0.0136.
\end{equation}
Terminating decimals below denote exact rational numbers.

Here is an integer calculation of the finite matrices. Put
$H_q=\{2^t\bmod q:1\le t\le\epsilon(q)\}$ and let $n_j(r)$ count
ordered sums of $j$ elements of $H_q$ congruent to $r\pmod q$.
With residues taken modulo $q$,
\begin{align*}
 n_0(r)&=\boldsymbol1_{r=0},&
 n_{j+1}(r)&=\sum_{h\in H_q}n_j(r-h),\\
 c_q(t)&=\sum_{d\mid(q,t)}d\mu(q/d),&
 A_{ij}(q)&=\frac{\displaystyle\sum_{r,s\bmod q}n_i(r)n_j(s)c_q(r-s)}
 {\varphi(q)^2\epsilon(q)^{i+j}}.
\end{align*}

The five matrices are
\begin{gather*}
 A(3)=\begin{pmatrix}
 \dfrac{1}{8} & -\dfrac{1}{16} & \dfrac{1}{32} \\[8pt]
 -\dfrac{1}{16} & \dfrac{1}{32} & -\dfrac{1}{64} \\[8pt]
 \dfrac{1}{32} & -\dfrac{1}{64} & \dfrac{1}{128}
 \end{pmatrix},\qquad
 A(5)=\begin{pmatrix}
 \dfrac{1}{64} & -\dfrac{1}{256} & \dfrac{1}{1024} \\[8pt]
 -\dfrac{1}{256} & \dfrac{1}{1024} & -\dfrac{1}{4096} \\[8pt]
 \dfrac{1}{1024} & -\dfrac{1}{4096} & \dfrac{1}{16384}
 \end{pmatrix},\\[8pt]
 A(7)=\begin{pmatrix}
 \dfrac{1}{27} & -\dfrac{1}{162} & -\dfrac{1}{162} \\[8pt]
 -\dfrac{1}{162} & \dfrac{2}{243} & -\dfrac{1}{729} \\[8pt]
 -\dfrac{1}{162} & -\dfrac{1}{729} & \dfrac{4}{2187}
 \end{pmatrix},\qquad
 A(15)=\begin{pmatrix}
 \dfrac{1}{32} & \dfrac{1}{256} & -\dfrac{7}{1024} \\[8pt]
 \dfrac{1}{256} & \dfrac{1}{128} & \dfrac{1}{1024} \\[8pt]
 -\dfrac{7}{1024} & \dfrac{1}{1024} & \dfrac{1}{512}
 \end{pmatrix},\\[8pt]
 A(31)=\begin{pmatrix}
 \dfrac{13}{2250} & \dfrac{1}{3750} & -\dfrac{8}{28125} \\[8pt]
 \dfrac{1}{3750} & \dfrac{77}{56250} & \dfrac{22}{140625} \\[8pt]
 -\dfrac{8}{28125} & \dfrac{22}{140625} & \dfrac{499}{1406250}
 \end{pmatrix}.
\end{gather*}

For the choices in Sections~\ref{sec:completion} and \ref{sec:five}, put
\[
 v=\left(-\frac1{16},-\frac14,1\right)^T,\qquad
 w=\left(-\frac18,1,0\right)^T.
\]
The diagonal entries and required quadratic forms are listed in
Table~\ref{tab:small-Gram}.
\begin{table}[htbp]
\centering
\renewcommand{\arraystretch}{1.4}
\caption{Diagonal entries and quadratic forms for
$v=(-1/16,-1/4,1)^T$ and $w=(-1/8,1,0)^T$.}
\begin{tabular}{crrrrr}
\toprule
$q$ & $A_{11}(q)$ & $A_{22}(q)$ & $A_{33}(q)$
& $v^TA(q)v$ & $\begin{gathered}w^TA(q)w\\+4v^TA(q)v\end{gathered}$\\
\midrule
$3$ & $\frac18$ & $\frac1{32}$ & $\frac1{128}$
& $\frac{25}{2048}$ & $\frac{25}{256}$\\
$5$ & $\frac1{64}$ & $\frac1{1024}$ & $\frac1{16384}$
& $\frac1{16384}$ & $\frac5{2048}$\\
$7$ & $\frac1{27}$ & $\frac2{243}$ & $\frac4{2187}$
& $\frac{2101}{559872}$ & $\frac{1775}{69984}$\\
$15$ & $\frac1{32}$ & $\frac1{128}$ & $\frac1{512}$
& $\frac{25}{8192}$ & $\frac5{256}$\\
$31$ & $\frac{13}{2250}$ & $\frac{77}{56250}$ & $\frac{499}{1406250}$
& $\frac{154309}{360000000}$ & $\frac{139817}{45000000}$\\
\bottomrule
\end{tabular}
\label{tab:small-Gram}
\end{table}

Summing the diagonal entries gives
\[
 (A_*)_{11}>0.2146,\qquad
 (A_*)_{22}>0.0496,\qquad
 (A_*)_{33}>0.012.
\]
The tail $T=A-A_*$ is positive semidefinite. By \eqref{eq:U},
\begin{align*}
 0\le T_{11}&<0.279-0.2146=0.0644<0.254^2,\\
 0\le T_{22}&<0.0555-0.0496=0.0059<0.077^2,\\
 0\le T_{33}&<0.0136-0.012=0.04^2,\\
 |T_{ij}|&\le\sqrt{T_{ii}T_{jj}}.
\end{align*}
Therefore, for $v=(v_1,v_2,v_3)^T\in\mathbb R^3$,
\begin{equation}\label{eq:PSDtail}
\begin{aligned}
 v^TAv&\le v^TA_*v+
 \left(\sum_{i=1}^3|v_i|\sqrt{T_{ii}}\right)^2\\
 &\le v^TA_*v+
 \bigl(0.254|v_1|+0.077|v_2|+0.04|v_3|\bigr)^2.
\end{aligned}
\end{equation}

\section{Major and minor arc estimates}\label{sec:arcs}

Fix a sufficiently small error parameter $\varepsilon>0$, and take
$\kappa$ sufficiently small in terms of $\varepsilon$.
The explicit formula of Pintz and Ruzsa~\cite{PRII}, applied with
$P_0=N^{4/9-1/450}$, gives a choice of $P=P(N)$ satisfying
\begin{equation}\label{eq:Prange}
 N^{11/25}\le P\le P_0.
\end{equation}
Write $\|t\|$ for the distance from $t$ to the nearest integer, and put
\begin{equation}\label{eq:arcs}
\begin{aligned}
 \mathfrak M&=\bigcup_{q\le P}\ \bigcup_{\substack{r\bmod q\\(r,q)=1}}
 \left\{\alpha\in(0,1]:\left\|\alpha-\frac rq\right\|\le\frac{P}{qN}\right\},\\
 \mathfrak m&=(0,1]\setminus\mathfrak M.
\end{aligned}
\end{equation}
The arcs are disjoint since $2P^2<N$. Define
\begin{equation}\label{eq:prime-integrals}
R_1(h)=\int_{\mathfrak M}|S(\alpha)|^2e(-h\alpha)\,\mathrm d\alpha,
\quad R(h)=\int_0^1|S(\alpha)|^2e(-h\alpha)\,\mathrm d\alpha.
\end{equation}
$R_1(h)$ and $R(h)$ are real and even, and
\begin{equation}\label{eq:Jtrivial}
0\le R_1(0)\le R(0)=\sum_{N_1<p\le N}(\log p)^2
=(1+o(1))N\log N.
\end{equation}

\begin{lemma}[{\cite[Theorem~2 and Corollary~1]{PRII}}]\label{lem:PR}
For this choice of $P$, there are $O_{\varepsilon}(1)$ integers
$r_\ell=r_\ell(N)$ and a fixed integer $c_{\varepsilon}$ such that
\begin{equation}\label{eq:PRestimate}
 |R_1(h)-N\mathfrak{S}(h)|\le\varepsilon N\mathfrak{S}(h)
\end{equation}
whenever $0<|h|\le3\cdot2^L$, $2\mid h$, and
$r_\ell\nmid c_{\varepsilon}h$ for every $\ell$.
The odd squarefree part of each $r_\ell$ tends to infinity.
For all even $h$ with $0<|h|\le3\cdot2^L$, including the exceptional
ones, one also has $|R_1(h)|\ll_{\varepsilon}N\mathfrak{S}(h)$.
\end{lemma}

\begin{lemma}\label{lem:mixedmajor}
For $1\le i,j\le3$, with the same $P$,
\begin{equation}\label{eq:mixedmajor}
\sum_{h\ne0}r_{i,j}(h)R_1(h)
=2NL^{i+j}\bigl(1+A_{ij}+O(\varepsilon)+o(1)\bigr).
\end{equation}
Consequently, for $Q,F$ in \eqref{eq:parameters} and
$v=(-b,-a,1)^T$,
\begin{equation}\label{eq:majorF}
\int_{\mathfrak M}|S(\alpha)\cdot F(G(\alpha))|^2\,\mathrm d\alpha
=2NL^6\bigl((1-a-b)^2+v^TAv+O(\varepsilon)+o(1)\bigr)
 +b^2L^5R_1(0).
\end{equation}
\end{lemma}

\begin{proof}
Put
\begin{gather*}
 \mathcal H=\{h\in2\mathbb Z:0<|h|\le3\cdot2^L,\quad
                  r_\ell\mid c_{\varepsilon}h\text{ for some }\ell\},\\
 q_\ell=\prod_{\substack{p\mid r_\ell\\p\nmid2c_{\varepsilon}}}p
 \longrightarrow\infty,\qquad
 r_\ell\mid c_{\varepsilon}h\ \Longrightarrow\ q_\ell\mid h.
\end{gather*}
Lemma~\ref{lem:tail} gives
\begin{equation*}
 \sum_{h\in\mathcal H}r_{i,j}(h)\mathfrak{S}(h)
 \le\sum_\ell\sum_{\substack{h\ne0\\q_\ell\mid h}}
 r_{i,j}(h)\mathfrak{S}(h)=o(L^{i+j}).
\end{equation*}
Lemma~\ref{lem:PR} and \eqref{eq:Blimit} now give
\begin{align*}
 &\left|\sum_{h\ne0}r_{i,j}(h)R_1(h)
       -N\sum_{h\ne0}r_{i,j}(h)\mathfrak{S}(h)\right|\\
 &\quad\le \varepsilon N\sum_{\substack{h\ne0\\h\notin\mathcal H}}
       r_{i,j}(h)\mathfrak{S}(h)
       +O_{\varepsilon}\left(N\sum_{h\in\mathcal H}
          r_{i,j}(h)\mathfrak{S}(h)\right)\\
 &\quad\le 2\varepsilon NL^{i+j}(1+A_{ij}+o(1))+o(NL^{i+j}),
\end{align*}
which proves \eqref{eq:mixedmajor}.

For $v=(v_1,v_2,v_3)^T=(-b,-a,1)^T$, we have
\begin{align*}
 &\int_{\mathfrak M}|S(\alpha)\cdot F(G(\alpha))|^2\,\mathrm d\alpha\\
 &\quad=\sum_{i,j=1}^3v_iv_jL^{6-i-j}
       \sum_h r_{i,j}(h)R_1(h)\\
 &\quad=2NL^6\left(\left(\sum_{i=1}^3v_i\right)^2
       +v^TAv+O(\varepsilon)+o(1)\right)
       +R_1(0)\sum_{i,j=1}^3v_iv_jL^{6-i-j}r_{i,j}(0).
\end{align*}
Finally, \eqref{eq:zerocount} and \eqref{eq:Jtrivial} give
\begin{align*}
 R_1(0)\sum_{i,j=1}^3v_iv_jL^{6-i-j}r_{i,j}(0)
 &=b^2L^5R_1(0)+O(L^4R_1(0))\\
 &=b^2L^5R_1(0)+o(NL^6).
\end{align*}
Substitute $\sum_i v_i=1-a-b$ to obtain \eqref{eq:majorF}.
\end{proof}

By \cite[Proposition~2]{JT} we have the uniform bound
\begin{equation}\label{eq:unweighted-sieve}
 \sum_{\substack{p_1,p_2\le N\\p_1-p_2=h}}1
 \le\frac{C_1+o(1)}2\frac{N\mathfrak{S}(h)}{(\log N)^2},
 \qquad C_1=6.7814,
\end{equation}
where the sum is over ordered pairs of primes and $h$ is a nonzero
even integer with $|h|<N$. The error is uniform in $h$; negative $h$
are covered by interchanging $p_1$ and $p_2$. Hence
\begin{equation}\label{eq:sieve}
\begin{aligned}
 R(h)&=\sum_{\substack{N_1<p_1,p_2\le N\\p_1-p_2=h}}
       (\log p_1)(\log p_2)\\
 &\le(\log N)^2\sum_{\substack{p_1,p_2\le N\\p_1-p_2=h}}1
 \le\frac{C_1+o(1)}2N\mathfrak{S}(h).
\end{aligned}
\end{equation}
Using $C_0R_0=1+A_{11}$, \eqref{eq:Blimit}, and
$\log N/L=\log2+o(1)$, we obtain
\begin{align}
\int_0^1|S(\alpha)\cdot G(\alpha)|^2\,\mathrm d\alpha
&=LR(0)+\sum_{h\ne0}r_{1,1}(h)R(h)
 \le NL^2\bigl(\log2+C_1C_0R_0+o(1)\bigr),\label{eq:fullbase}\\
\int_{\mathfrak M}|S(\alpha)\cdot G(\alpha)|^2\,\mathrm d\alpha
&=LR_1(0)+2NL^2\bigl(C_0R_0+O(\varepsilon)+o(1)\bigr).\label{eq:majorbase}
\end{align}
Subtracting \eqref{eq:majorbase} from \eqref{eq:fullbase}, we obtain
\begin{equation}\label{eq:minorbase}
\begin{aligned}
\int_{\mathfrak m}|S(\alpha)\cdot G(\alpha)|^2\,\mathrm d\alpha
&=\int_0^1|S(\alpha)\cdot G(\alpha)|^2\,\mathrm d\alpha
 -\int_{\mathfrak M}|S(\alpha)\cdot G(\alpha)|^2\,\mathrm d\alpha\\
&\le2NL^2\left(\frac{C_1-2}{2}C_0R_0+\frac{\log2}{2}
 +O(\varepsilon)+o(1)\right)-LR_1(0).
\end{aligned}
\end{equation}

\begin{lemma}\label{lem:Jlower}
Uniformly for $P$ satisfying \eqref{eq:Prange},
\begin{equation}\label{eq:Jlower}
 R_1(0) = (1+o(1)) N\log P.
\end{equation}
\end{lemma}

\begin{proof} This lemma is similar to \cite[(4.8)]{PRII}, in which the range of $P$ is smaller.  
Since $p>N_1>P\ge q$, the Ramanujan sum identity and Cauchy's inequality give
\begin{equation*}
 \sum_{\substack{a\bmod q\\(a,q)=1}}|S(a/q+\beta)|^2
 \ge\frac1{\varphi(q)}
 \left|\sum_{\substack{a\bmod q\\(a,q)=1}}S(a/q+\beta)\right|^2
 =\frac{\mu(q)^2}{\varphi(q)}|S(\beta)|^2.
\end{equation*}
For fixed $T\ge1$,
\[
 q\le P/T\ \Longrightarrow\ \frac{T}{N}\le\frac{P}{qN},
 \qquad
 \sum_{q\le P/T}\frac{\mu(q)^2}{\varphi(q)}=\log P+O_T(1).
\]
The disjointness of the arcs gives
\begin{align*}
 R_1(0)&\ge\sum_{q\le P/T}\frac{\mu(q)^2}{\varphi(q)}
                 \int_{-T/N}^{T/N}|S(\beta)|^2\,\mathrm d\beta\\
 &=\bigl(\log P+O_T(1)\bigr)
       \int_{-T/N}^{T/N}|S(\beta)|^2\,\mathrm d\beta.
\end{align*}
By the prime number theorem and partial summation, uniformly for $|u|\le T$,
\[
 \frac{S(u/N)}N=\int_0^1e(ut)\,\mathrm dt+o_T(1).
\]
Together with Plancherel's theorem, this gives
\begin{equation}\label{eq:central-integral}
 \frac1N\int_{-T/N}^{T/N}|S(\beta)|^2\,\mathrm d\beta
 =\int_{-T}^{T}\left|\int_0^1e(ut)\,\mathrm dt\right|^2\,\mathrm du+o_T(1)
 =1+O(T^{-1})+o_T(1).
\end{equation}
Here
\[
 \int_{|u|>T}\left|\int_0^1e(ut)\,\mathrm dt\right|^2\,\mathrm du
 \le\int_{|u|>T}\frac{\mathrm du}{\pi^2u^2}=\frac{2}{\pi^2T}.
\]
Thus
\[
 \liminf_{N\to\infty}\frac{R_1(0)}{N\log P}
 \ge1-\frac{2}{\pi^2T}\qquad(T\ge1).
\]
Letting $T\to\infty$ proves \eqref{eq:Jlower}.
\end{proof}

Under GRH we use the arcs of Pintz and Ruzsa~\cite{PRI}.

\begin{lemma}\label{lem:grhbridge}
Assume GRH and fix $0<\kappa<1/2$. Put
$L_0=\lfloor\log_2N\rfloor$, $P=N^{1/2}L_0^{-8}$, $Q_0=N^{1/2}$, and
\begin{align*}
 \mathfrak M_{\mathrm G}
 &=\bigcup_{q\le P}\ \bigcup_{\substack{a\bmod q\\(a,q)=1}}
 \left\{\alpha\in[0,1]:\left\|\alpha-\frac aq\right\|\le\frac1{qQ_0}\right\},\\
 \mathfrak m_{\mathrm G}&=[0,1]\setminus\mathfrak M_{\mathrm G},\qquad
 R_{1,\mathrm G}(h)=\int_{\mathfrak M_{\mathrm G}}|S(\alpha)|^2e(-h\alpha)\,\mathrm d\alpha.
\end{align*}
For $1\le i,j\le3$,
\begin{align}
 \int_{\mathfrak M_{\mathrm G}}|S(\alpha)|^2G(\alpha)^i
 \overline{G(\alpha)^j}\,\mathrm d\alpha
 &=2NL^{i+j}\bigl(1+A_{ij}+o(1)\bigr)+r_{i,j}(0)R_{1,\mathrm G}(0),
 \label{eq:grhmixedbridge}\\
 R_{1,\mathrm G}(0)&\ge\bigl(\tfrac12-o(1)\bigr)N\log N.
 \label{eq:grhJbridge}
\end{align}
Moreover,
\begin{align}
 \int_{\mathfrak m_{\mathrm G}}|S(\alpha)\cdot G(\alpha)|^2\,\mathrm d\alpha
 &\le2NL^2\left(\frac{C_1-2}{2}C_0R_0+\frac{\log2}{2}+o(1)\right)-LR_{1,\mathrm G}(0)
 \notag\\
 &\le2NL^2\left(\frac{C_1-2}{2}C_0R_0+\frac{\log2}{4}+o(1)\right).
 \label{eq:grhminorbridge}
\end{align}
Let $\mathcal E=\mathcal E(N)\subseteq[0,1]$ be measurable. If
$\int_{\mathcal E}\mathrm d\alpha\ll N^{-\delta+o(1)}$ for a fixed
$\delta>1/2$, then
\begin{equation}\label{eq:grhexceptionalbridge}
 \int_{\mathfrak m_{\mathrm G}\cap\mathcal E}|S(\alpha)|^2\,\mathrm d\alpha=o(N).
\end{equation}
\end{lemma}

\begin{proof}
The arcs are disjoint since $2P<Q_0$. Put
\begin{equation*}
 S_0(\alpha)=\sum_{2<p\le N}e(p\alpha),\qquad
 W(\alpha)=(\log N)\cdot S_0(\alpha).
\end{equation*}
Write
$\|V\|_2^2=\int_0^1|V(\alpha)|^2\,\mathrm d\alpha$, and let
$\pi(x)$ count the primes not exceeding $x$.
The sum in \cite{PRI} also includes $2$. For any measurable $E\subseteq[0,1]$,
\[
 \left|\int_E\bigl(|S_0(\alpha)+e(2\alpha)|^2-|S_0(\alpha)|^2\bigr)
 e(-h\alpha)\,\mathrm d\alpha\right|
 \le2\|S_0\|_2+1\ll N^{1/2}.
\]
Thus deleting $2$ is absorbed by the error below.
Pintz and Ruzsa~\cite[Lemma~1, pp.~170--171]{PRI} proved that,
uniformly for nonzero even $h$,
\begin{equation*}
 \int_{\mathfrak M_{\mathrm G}}|S_0(\alpha)|^2e(-h\alpha)\,\mathrm d\alpha
 =\frac{\mathfrak{S}(h)\max(N-|h|,0)}{(\log N)^2}+O(NL_0^{-3}).
\end{equation*}
Using \eqref{eq:Blimit} and
\[
 r_{i,j}(h)\ne0\ \Longrightarrow\ |h|\le3\cdot2^L=o(N/L),
 \qquad \sum_h r_{i,j}(h)=L^{i+j},
\]
we obtain
\begin{align*}
 &\int_{\mathfrak M_{\mathrm G}}|W(\alpha)|^2
 G(\alpha)^i\overline{G(\alpha)^j}\,\mathrm d\alpha
 -r_{i,j}(0)\int_{\mathfrak M_{\mathrm G}}|W(\alpha)|^2\,\mathrm d\alpha\\
 &\quad=\sum_{h\ne0}r_{i,j}(h)
       \bigl((N-|h|)\mathfrak{S}(h)+O(N/L)\bigr)\\
 &\quad=N\sum_{h\ne0}r_{i,j}(h)\mathfrak{S}(h)
       +O\left(2^L\sum_{h\ne0}r_{i,j}(h)\mathfrak{S}(h)
                    +NL^{i+j-1}\right)\\
 &\quad=2NL^{i+j}\bigl(1+A_{ij}+o(1)\bigr),
\end{align*}
as $N\to\infty$.

Parseval and partial summation give
\begin{equation}\label{eq:weighttransfererror}
 \|S-W\|_2^2
 \le\sum_{N_1<p\le N}\bigl(\log(N/p)\bigr)^2
       +\pi(N_1)(\log N)^2
 \ll\frac N{\log N}+N_1\log N
 \ll\frac N{\log N}.
\end{equation}
For $1\le i\le3$, \eqref{eq:unweighted-sieve} gives
\begin{align*}
 \|WG^i\|_2^2
 &=(\log N)^2\sum_h r_{i,i}(h)
       \sum_{\substack{2<p_1,p_2\le N\\p_1-p_2=h}}1\\
 &\le(\log N)^2(\pi(N)-1)r_{i,i}(0)
       +\frac{C_1+o(1)}2N\sum_{h\ne0}r_{i,i}(h)\mathfrak{S}(h)\\
 &\ll NL\cdot r_{i,i}(0)+NL^{2i}\ll NL^{2i},
\end{align*}
where the last line uses \eqref{eq:zerocount} and \eqref{eq:Blimit}.
Also, $|G(\alpha)|\le L$ and \eqref{eq:weighttransfererror} imply
\begin{equation*}
 \|(S-W)G^i\|_2^2\le L^{2i}\|S-W\|_2^2
 \ll\frac{NL^{2i}}{\log N}.
\end{equation*}
For every measurable $\mathcal A\subseteq[0,1]$, the identity
\begin{equation*}
 |S(\alpha)|^2-|W(\alpha)|^2
 =(S(\alpha)-W(\alpha))\overline{W(\alpha)}
  +W(\alpha)\overline{S(\alpha)-W(\alpha)}
  +|S(\alpha)-W(\alpha)|^2
\end{equation*}
and Cauchy's inequality give
\begin{align*}
 &\left|\int_{\mathcal A}(|S(\alpha)|^2-|W(\alpha)|^2)
 G(\alpha)^i\overline{G(\alpha)^j}\,\mathrm d\alpha\right|\\
 &\quad\le\|(S-W)G^i\|_2\|WG^j\|_2
       +\|WG^i\|_2\|(S-W)G^j\|_2\\
 &\qquad\quad+\|(S-W)G^i\|_2\|(S-W)G^j\|_2\\
 &\quad\ll\frac{NL^{i+j}}{\sqrt{\log N}}=o(NL^{i+j}).
\end{align*}
For the zero difference, $\|W\|_2^2=(\log N)^2(\pi(N)-1)\ll N\log N$,
so
\begin{align*}
 \left|R_{1,\mathrm G}(0)
       -\int_{\mathfrak M_{\mathrm G}}|W(\alpha)|^2\,\mathrm d\alpha\right|
 &\le2\|S-W\|_2\|W\|_2+\|S-W\|_2^2\ll N,\\
 N\cdot r_{i,j}(0)&\ll NL^{i+j-1}=o(NL^{i+j}).
\end{align*}
These estimates prove \eqref{eq:grhmixedbridge}.

For fixed $T\ge1$, all the arcs contain $|\beta|\le T/N$ once $N$
is large, since
\begin{equation*}
 \frac{N}{qQ_0}\ge\frac{N}{PQ_0}=L_0^8\ge T,
 \qquad N_1>P.
\end{equation*}
The Ramanujan sum estimate in Lemma~\ref{lem:Jlower} gives
\begin{align*}
 R_{1,\mathrm G}(0)&\ge
 \left(\sum_{q\le P}\frac{\mu(q)^2}{\varphi(q)}\right)
 \int_{-T/N}^{T/N}|S(\beta)|^2\,\mathrm d\beta\\
 &=N(\log P+O(1))\bigl(1+O(T^{-1})+o_T(1)\bigr).
\end{align*}
Since
\[
 \frac{\log P}{\log N}=\frac12-\frac{8\log L_0}{\log N}
 =\frac12+o(1),
\]
we have $\liminf_{N\to\infty}R_{1,\mathrm G}(0)/(N\log N)
\ge\tfrac12-O(T^{-1})$. Letting $T\to\infty$ proves
\eqref{eq:grhJbridge}.
For the minor arcs, \eqref{eq:fullbase} and the case $i=j=1$ of
\eqref{eq:grhmixedbridge} give
\begin{align*}
 \int_{\mathfrak m_{\mathrm G}}|S(\alpha)\cdot G(\alpha)|^2\,\mathrm d\alpha
 &=\int_0^1|S(\alpha)\cdot G(\alpha)|^2\,\mathrm d\alpha
   -\int_{\mathfrak M_{\mathrm G}}|S(\alpha)\cdot G(\alpha)|^2\,\mathrm d\alpha\\
 &\le NL^2\bigl(\log2+(C_1-2)C_0R_0\bigr)
      -LR_{1,\mathrm G}(0)+o(NL^2).
\end{align*}
By \eqref{eq:grhJbridge},
\[
 \frac{R_{1,\mathrm G}(0)}{NL}\ge\frac{\log2}{2}+o(1),
\]
which gives \eqref{eq:grhminorbridge}.

Finally, Pintz and Ruzsa~\cite[Lemma~2, p.~171]{PRI} give
\begin{equation*}
 \sup_{\alpha\in\mathfrak m_{\mathrm G}}|S_0(\alpha)|
 \ll N^{3/4}L_0^2,\qquad
 \sup_{\alpha\in\mathfrak m_{\mathrm G}}|W(\alpha)|^2
 \ll N^{3/2}L_0^6.
\end{equation*}
Thus, for the set $\mathcal E$ in the statement,
\begin{align*}
 \int_{\mathfrak m_{\mathrm G}\cap\mathcal E}|S(\alpha)|^2\,\mathrm d\alpha
 &\le2\left(\int_{\mathcal E}\mathrm d\beta\right)
       \sup_{\alpha\in\mathfrak m_{\mathrm G}}|W(\alpha)|^2
       +2\|S-W\|_2^2\\
 &\ll N^{3/2}L_0^6\int_{\mathcal E}\mathrm d\alpha+N/\log N\\
 &\ll N^{3/2-\delta+o(1)}+N/\log N=o(N),
\end{align*}
since $\delta>1/2$.
\end{proof}

\section{Large deviations of the dyadic sum}\label{sec:large}

Recall that
\begin{equation*}
 z_L(\alpha)=\frac1L\sum_{\nu=1}^L e(2^\nu\alpha).
\end{equation*}
We estimate its real and imaginary parts together using an
exponential moment bound.

Pintz and Ruzsa~\cite[Theorem~2 and Sections~3--7]{PRI} studied
exponential moments of the dyadic sum and obtained explicit bounds
for the set on which its modulus is large.
Aistleitner, Gantert, Kabluchko, Prochno and
Ramanan~\cite[Theorem~B and Proposition~3.4]{AGKPR} proved a
large-deviation principle for geometric cosine sums.
Fr\"uhwirth, Juhos and Prochno~\cite[Theorem~B]{FJP} treated sums
of real periodic functions along integer geometric sequences.

\subsection{An exponential identity}

\begin{lemma}\label{lem:exponential-identity}
Let $w$ be a real continuous $1$-periodic function such that
$w(t+1/2)=-w(t)$. For every positive integer $L$,
\begin{equation}\label{eq:exponential-identity}
 \int_0^1\exp\left(\sum_{\nu=1}^L
 [w(2^\nu\alpha)-\log\cosh w(2^\nu\alpha)]\right)
 \,\mathrm d\alpha=1,
\end{equation}
where $\cosh t=(e^t+e^{-t})/2$.
\end{lemma}

\begin{proof}
Put
\begin{equation*}
 h(t)=\frac{\exp(w(t))}{\cosh w(t)},\qquad
 I_\ell=\int_0^1\prod_{\nu=0}^{\ell-1}h(2^\nu\alpha)
 \,\mathrm d\alpha,\qquad I_0=1.
\end{equation*}
Then $h(t)+h(t+1/2)=2$. By periodicity and the changes of variable
$\alpha\mapsto\alpha+1/2$ and $t=2\alpha$, for $\ell\ge1$,
\begin{align*}
 I_\ell
 &=\frac12\int_0^1\bigl(h(\alpha)+h(\alpha+1/2)\bigr)
       \prod_{\nu=1}^{\ell-1}h(2^\nu\alpha)\,\mathrm d\alpha\\
 &=\int_0^1\prod_{\nu=1}^{\ell-1}h(2^\nu\alpha)\,\mathrm d\alpha
 =\frac12\int_0^2\prod_{\nu=0}^{\ell-2}h(2^\nu t)\,\mathrm dt
 =I_{\ell-1}.
\end{align*}
Thus $I_L=1$, and another change $t=2\alpha$ gives
\begin{equation*}
 \int_0^1\prod_{\nu=1}^Lh(2^\nu\alpha)\,\mathrm d\alpha
 =\frac12\int_0^2\prod_{\nu=0}^{L-1}h(2^\nu t)\,\mathrm dt
 =I_L=1.
\end{equation*}
\end{proof}

\subsection{Second moments and the region}

For fixed $\xi,\eta\in\mathbb R$, put
\begin{equation*}
 \begin{aligned}
 w_{\xi,\eta}(t)&=\xi\Re e(t)+\eta\Im e(t),\\
 V_{\xi,\eta}(x,y)
 &=\frac{(1+x)\xi^2+2y\xi\eta+(1-x)\eta^2}{2}.
 \end{aligned}
\end{equation*}
For $z=x+iy$ with $|z|\le1$,
\begin{equation*}
 2V_{\xi,\eta}(x,y)
 =\xi^2+\eta^2+x(\xi^2-\eta^2)+2y\xi\eta
 \ge(1-|z|)(\xi^2+\eta^2)\ge0.
\end{equation*}
Since
\begin{equation*}
 w_{\xi,\eta}(t)^2
 =\frac{\xi^2+\eta^2}{2}
  +\frac12\Re\bigl((\xi-i\eta)^2e(2t)\bigr),
\end{equation*}
writing $z_L(\alpha)=x+iy$ gives
\begin{equation}\label{eq:directional-second-moment}
 \begin{aligned}
 &\frac1L\sum_{\nu=1}^L w_{\xi,\eta}(2^\nu\alpha)^2
       -V_{\xi,\eta}(x,y)\\
 &\qquad=\frac1{2L}\Re\left((\xi-i\eta)^2
       \bigl(e(2^{L+1}\alpha)-e(2\alpha)\bigr)\right).
 \end{aligned}
\end{equation}
Consequently,
\begin{equation*}
 \left|\frac1L\sum_{\nu=1}^Lw_{\xi,\eta}(2^\nu\alpha)^2
       -V_{\xi,\eta}(x,y)\right|
 \le\frac{\xi^2+\eta^2}{L}.
\end{equation*}

For $t>0$, differentiation gives
\begin{equation*}
 \tanh t-\frac{t}{\cosh^2t}
 =\int_0^t\frac{2v\tanh v}{\cosh^2v}\,\mathrm dv\ge0,
 \qquad
 \frac{\mathrm d}{\mathrm dt}\frac{\tanh t}{t}\le0.
\end{equation*}
It follows, by integrating the decreasing derivative
$\tanh\sqrt u/(2\sqrt u)$, that for $u,v\ge0$,
\begin{equation*}
 \log\cosh\sqrt u
 \le\log\cosh\sqrt v+
       \frac{\tanh\sqrt v}{2\sqrt v}(u-v),\qquad
 0\le\frac{\tanh\sqrt v}{2\sqrt v}\le\frac12,
\end{equation*}
where the quotient at $v=0$ is $1/2$.
Taking $v=V_{\xi,\eta}(x,y)$ and summing with
$u=w_{\xi,\eta}(2^\nu\alpha)^2$, we obtain
\begin{equation*}
 \sum_{\nu=1}^L\log\cosh w_{\xi,\eta}(2^\nu\alpha)
 \le L\log\cosh\sqrt{V_{\xi,\eta}(x,y)}
       +\frac{\xi^2+\eta^2}{2}.
\end{equation*}
For $|z|\le1$, define
\begin{equation}\label{eq:directional-cost}
 \Phi_{\xi,\eta}(z)=\xi x+\eta y
                 -\log\cosh\sqrt{V_{\xi,\eta}(x,y)}.
\end{equation}
Since $L(\xi x+\eta y)=\sum_{\nu=1}^Lw_{\xi,\eta}(2^\nu\alpha)$,
\begin{equation*}
 L\Phi_{\xi,\eta}(z_L(\alpha))
 \le\sum_{\nu=1}^L
 [w_{\xi,\eta}(2^\nu\alpha)-\log\cosh w_{\xi,\eta}(2^\nu\alpha)]
 +\frac{\xi^2+\eta^2}{2}.
\end{equation*}
Lemma~\ref{lem:exponential-identity} therefore gives
\begin{equation}\label{eq:entropy-moment}
 \int_0^1\exp\bigl(L\Phi_{\xi,\eta}(z_L(\alpha))\bigr)
 \,\mathrm d\alpha
 \le\exp\left(\frac{\xi^2+\eta^2}{2}\right).
\end{equation}

Let $s>0$, $z=x+iy$ and $0<u=x^2+y^2<1$. Choose
\begin{equation}\label{eq:chosen-coefficients}
 \xi=\frac{\sqrt2\,s(x-u)}{\sqrt{u(1-x)(1-u)}},\qquad
 \eta=\frac{\sqrt2\,s y}{\sqrt{u(1-x)(1-u)}}.
\end{equation}
The identities
\begin{align*}
 (1+x)(x-u)^2+2y^2(x-u)+(1-x)y^2
 &=u(1-x)(1-u),\\
 x(x-u)+y^2&=u(1-x)
\end{align*}
give
\begin{equation}\label{eq:directional-value}
 V_{\xi,\eta}(x,y)=s^2,\qquad
 \Phi_{\xi,\eta}(z)
 =s\sqrt{\frac{2u(1-x)}{1-u}}-\log\cosh s.
\end{equation}

For $0<\lambda<1$, take $s=\frac12\log\frac{1+\lambda}{1-\lambda}$.
Then $\cosh s=(1-\lambda^2)^{-1/2}$ and
\begin{equation}\label{eq:entropy}
 \begin{aligned}
 I(\lambda)&:=\frac{1+\lambda}{2}\log(1+\lambda)
             +\frac{1-\lambda}{2}\log(1-\lambda)\\
 &=s\lambda-\log\cosh s.
 \end{aligned}
\end{equation}
We use $I(\lambda)$ as an exponent supplied by
\eqref{eq:entropy-moment}, without asserting that it is the best
possible exponent for the dyadic sum.

Put
\begin{equation}\label{eq:region-parameters}
 r_\lambda=\frac{\lambda^2}{2},\qquad
 \beta_\lambda=
 \frac{r_\lambda+\sqrt{r_\lambda^2+4r_\lambda}}2.
\end{equation}
Thus
\begin{equation*}
 \beta_\lambda^2=r_\lambda(1+\beta_\lambda),\qquad
 0<r_\lambda<\frac12,\qquad 0<\beta_\lambda<1.
\end{equation*}
Define
\begin{equation}\label{eq:region}
 \mathcal D_\lambda=
 \left\{x+iy:x\le\beta_\lambda,\quad
 x^2+y^2\le\frac{r_\lambda}{1+r_\lambda-x}\right\}.
\end{equation}
For $x\le\beta_\lambda$,
\begin{equation*}
 1+r_\lambda-x\ge1+r_\lambda-\beta_\lambda
 =\frac1{1+\beta_\lambda}>0,\qquad
 \frac{r_\lambda}{1+r_\lambda-x}\le\beta_\lambda^2<1.
\end{equation*}
Hence $\mathcal D_\lambda$ is closed and bounded, and
$|z|\le\beta_\lambda$ on this region. Conversely, for
$v=\sqrt{x^2+y^2}<1$,
\begin{align*}
 (1+r_\lambda-x)(x^2+y^2)\le r_\lambda
 &\Longrightarrow
 v^2(1-v)\le v^2(1-x)\le r_\lambda(1-v^2)\\
 &\Longrightarrow v^2\le r_\lambda(1+v)
 \Longrightarrow v\le\beta_\lambda.
\end{align*}
It follows that
\begin{equation}\label{eq:region-test}
 \begin{aligned}
 |z|<1:\quad
 z\in\mathcal D_\lambda
 &\Longleftrightarrow (1+r_\lambda-x)|z|^2\le r_\lambda,\\
 |z|=1:\quad
 (1+r_\lambda-x)|z|^2-r_\lambda&=1-x\ge0.
 \end{aligned}
\end{equation}
Figure~\ref{fig:entropy-regions}
shows the two regions used below.

\begin{figure}[htbp]
\centering
\begin{tikzpicture}[x=2.9cm,y=2.9cm,font=\small]
\fill[blue!8] plot coordinates {
(0.787910,0.000000) (0.787000,0.020608) (0.784310,0.041104) (0.779952,0.061384) (0.774086,0.081360)
(0.766895,0.100964) (0.758563,0.120145) (0.749262,0.138868) (0.739141,0.157109) (0.728328,0.174856)
(0.716929,0.192101) (0.705032,0.208840) (0.692709,0.225075) (0.680020,0.240808) (0.667013,0.256042)
(0.653730,0.270784) (0.640205,0.285037) (0.626467,0.298809) (0.612541,0.312105) (0.598449,0.324931)
(0.584209,0.337293) (0.569839,0.349198) (0.555352,0.360650) (0.540762,0.371655) (0.526081,0.382220)
(0.511320,0.392349) (0.496489,0.402049) (0.481598,0.411323) (0.466655,0.420178) (0.451669,0.428618)
(0.436647,0.436647) (0.421598,0.444272) (0.406528,0.451495) (0.391444,0.458322) (0.376353,0.464757)
(0.361261,0.470804) (0.346174,0.476468) (0.331100,0.481753) (0.316042,0.486663) (0.301008,0.491201)
(0.286003,0.495373) (0.271033,0.499181) (0.256103,0.502630) (0.241218,0.505724) (0.226384,0.508467)
(0.211606,0.510862) (0.196889,0.512913) (0.182238,0.514624) (0.167658,0.516000) (0.153155,0.517043)
(0.138733,0.517757) (0.124396,0.518147) (0.110150,0.518216) (0.096000,0.517968) (0.081949,0.517406)
(0.068003,0.516535) (0.054166,0.515359) (0.040443,0.513880) (0.026838,0.512104) (0.013356,0.510034)
(0.000000,0.507673) (-0.013225,0.505026) (-0.026314,0.502098) (-0.039264,0.498890) (-0.052070,0.495409)
(-0.064728,0.491657) (-0.077235,0.487640) (-0.089585,0.483360) (-0.101777,0.478822) (-0.113805,0.474030)
(-0.125665,0.468989) (-0.137355,0.463702) (-0.148870,0.458174) (-0.160207,0.452410) (-0.171362,0.446412)
(-0.182331,0.440187) (-0.193112,0.433737) (-0.203701,0.427068) (-0.214094,0.420184) (-0.224289,0.413089)
(-0.234282,0.405788) (-0.244070,0.398286) (-0.253650,0.390586) (-0.263019,0.382694) (-0.272174,0.374615)
(-0.281112,0.366352) (-0.289831,0.357911) (-0.298327,0.349296) (-0.306599,0.340512) (-0.314643,0.331565)
(-0.322457,0.322457) (-0.330039,0.313196) (-0.337386,0.303784) (-0.344497,0.294228) (-0.351368,0.284532)
(-0.357998,0.274701) (-0.364384,0.264741) (-0.370525,0.254655) (-0.376419,0.244449) (-0.382064,0.234129)
(-0.387458,0.223699) (-0.392600,0.213164) (-0.397487,0.202530) (-0.402119,0.191801) (-0.406494,0.180983)
(-0.410611,0.170081) (-0.414468,0.159099) (-0.418064,0.148044) (-0.421398,0.136921) (-0.424470,0.125734)
(-0.427277,0.114489) (-0.429820,0.103191) (-0.432097,0.091845) (-0.434108,0.080457) (-0.435852,0.069032)
(-0.437329,0.057575) (-0.438537,0.046092) (-0.439478,0.034588) (-0.440150,0.023067) (-0.440553,0.011536)
(-0.440688,0.000000) (-0.440553,-0.011536) (-0.440150,-0.023067) (-0.439478,-0.034588) (-0.438537,-0.046092)
(-0.437329,-0.057575) (-0.435852,-0.069032) (-0.434108,-0.080457) (-0.432097,-0.091845) (-0.429820,-0.103191)
(-0.427277,-0.114489) (-0.424470,-0.125734) (-0.421398,-0.136921) (-0.418064,-0.148044) (-0.414468,-0.159099)
(-0.410611,-0.170081) (-0.406494,-0.180983) (-0.402119,-0.191801) (-0.397487,-0.202530) (-0.392600,-0.213164)
(-0.387458,-0.223699) (-0.382064,-0.234129) (-0.376419,-0.244449) (-0.370525,-0.254655) (-0.364384,-0.264741)
(-0.357998,-0.274701) (-0.351368,-0.284532) (-0.344497,-0.294228) (-0.337386,-0.303784) (-0.330039,-0.313196)
(-0.322457,-0.322457) (-0.314643,-0.331565) (-0.306599,-0.340512) (-0.298327,-0.349296) (-0.289831,-0.357911)
(-0.281112,-0.366352) (-0.272174,-0.374615) (-0.263019,-0.382694) (-0.253650,-0.390586) (-0.244070,-0.398286)
(-0.234282,-0.405788) (-0.224289,-0.413089) (-0.214094,-0.420184) (-0.203701,-0.427068) (-0.193112,-0.433737)
(-0.182331,-0.440187) (-0.171362,-0.446412) (-0.160207,-0.452410) (-0.148870,-0.458174) (-0.137355,-0.463702)
(-0.125665,-0.468989) (-0.113805,-0.474030) (-0.101777,-0.478822) (-0.089585,-0.483360) (-0.077235,-0.487640)
(-0.064728,-0.491657) (-0.052070,-0.495409) (-0.039264,-0.498890) (-0.026314,-0.502098) (-0.013225,-0.505026)
(-0.000000,-0.507673) (0.013356,-0.510034) (0.026838,-0.512104) (0.040443,-0.513880) (0.054166,-0.515359)
(0.068003,-0.516535) (0.081949,-0.517406) (0.096000,-0.517968) (0.110150,-0.518216) (0.124396,-0.518147)
(0.138733,-0.517757) (0.153155,-0.517043) (0.167658,-0.516000) (0.182238,-0.514624) (0.196889,-0.512913)
(0.211606,-0.510862) (0.226384,-0.508467) (0.241218,-0.505724) (0.256103,-0.502630) (0.271033,-0.499181)
(0.286003,-0.495373) (0.301008,-0.491201) (0.316042,-0.486663) (0.331100,-0.481753) (0.346174,-0.476468)
(0.361261,-0.470804) (0.376353,-0.464757) (0.391444,-0.458322) (0.406528,-0.451495) (0.421598,-0.444272)
(0.436647,-0.436647) (0.451669,-0.428618) (0.466655,-0.420178) (0.481598,-0.411323) (0.496489,-0.402049)
(0.511320,-0.392349) (0.526081,-0.382220) (0.540762,-0.371655) (0.555352,-0.360650) (0.569839,-0.349198)
(0.584209,-0.337293) (0.598449,-0.324931) (0.612541,-0.312105) (0.626467,-0.298809) (0.640205,-0.285037)
(0.653730,-0.270784) (0.667013,-0.256042) (0.680020,-0.240808) (0.692709,-0.225075) (0.705032,-0.208840)
(0.716929,-0.192101) (0.728328,-0.174856) (0.739141,-0.157109) (0.749262,-0.138868) (0.758563,-0.120145)
(0.766895,-0.100964) (0.774086,-0.081360) (0.779952,-0.061384) (0.784310,-0.041104) (0.787000,-0.020608)
(0.787910,-0.000000)
} -- cycle;
\fill[blue!22] plot coordinates {
(0.725701,0.000000) (0.725037,0.018986) (0.723062,0.037894) (0.719821,0.056651) (0.715387,0.075190)
(0.709849,0.093453) (0.703304,0.111392) (0.695853,0.128969) (0.687594,0.146153) (0.678620,0.162922)
(0.669013,0.179262) (0.658849,0.195160) (0.648193,0.210611) (0.637103,0.225610) (0.625630,0.240157)
(0.613816,0.254251) (0.601701,0.267895) (0.589318,0.281090) (0.576696,0.293842) (0.563861,0.306152)
(0.550836,0.318025) (0.537640,0.329466) (0.524292,0.340479) (0.510808,0.351069) (0.497204,0.361240)
(0.483492,0.370997) (0.469685,0.380344) (0.455795,0.389285) (0.441831,0.397827) (0.427805,0.405972)
(0.413724,0.413724) (0.399599,0.421090) (0.385437,0.428071) (0.371247,0.434674) (0.357035,0.440901)
(0.342809,0.446757) (0.328576,0.452246) (0.314342,0.457371) (0.300115,0.462136) (0.285900,0.466547)
(0.271704,0.470605) (0.257532,0.474315) (0.243390,0.477681) (0.229285,0.480706) (0.215221,0.483394)
(0.201204,0.485748) (0.187239,0.487773) (0.173331,0.489472) (0.159487,0.490849) (0.145709,0.491907)
(0.132005,0.492650) (0.118378,0.493081) (0.104834,0.493204) (0.091377,0.493024) (0.078011,0.492542)
(0.064742,0.491764) (0.051574,0.490693) (0.038511,0.489332) (0.025558,0.487685) (0.012720,0.485756)
(0.000000,0.483549) (-0.012597,0.481067) (-0.025067,0.478315) (-0.037407,0.475296) (-0.049611,0.472013)
(-0.061675,0.468471) (-0.073597,0.464674) (-0.085372,0.460626) (-0.096996,0.456330) (-0.108465,0.451790)
(-0.119776,0.447011) (-0.130926,0.441997) (-0.141909,0.436751) (-0.152724,0.431278) (-0.163366,0.425582)
(-0.173832,0.419667) (-0.184119,0.413537) (-0.194223,0.407197) (-0.204142,0.400650) (-0.213871,0.393902)
(-0.223409,0.386955) (-0.232751,0.379816) (-0.241896,0.372487) (-0.250840,0.364974) (-0.259580,0.357281)
(-0.268113,0.349412) (-0.276437,0.341372) (-0.284550,0.333165) (-0.292448,0.324796) (-0.300129,0.316270)
(-0.307591,0.307591) (-0.314832,0.298764) (-0.321848,0.289794) (-0.328639,0.280684) (-0.335201,0.271441)
(-0.341533,0.262068) (-0.347633,0.252570) (-0.353499,0.242953) (-0.359129,0.233221) (-0.364521,0.223379)
(-0.369673,0.213431) (-0.374585,0.203383) (-0.379253,0.193239) (-0.383678,0.183005) (-0.387858,0.172685)
(-0.391790,0.162285) (-0.395475,0.151809) (-0.398911,0.141262) (-0.402096,0.130649) (-0.405031,0.119976)
(-0.407713,0.109246) (-0.410142,0.098466) (-0.412318,0.087641) (-0.414239,0.076775) (-0.415905,0.065873)
(-0.417316,0.054941) (-0.418471,0.043983) (-0.419370,0.033005) (-0.420012,0.022012) (-0.420397,0.011008)
(-0.420525,0.000000) (-0.420397,-0.011008) (-0.420012,-0.022012) (-0.419370,-0.033005) (-0.418471,-0.043983)
(-0.417316,-0.054941) (-0.415905,-0.065873) (-0.414239,-0.076775) (-0.412318,-0.087641) (-0.410142,-0.098466)
(-0.407713,-0.109246) (-0.405031,-0.119976) (-0.402096,-0.130649) (-0.398911,-0.141262) (-0.395475,-0.151809)
(-0.391790,-0.162285) (-0.387858,-0.172685) (-0.383678,-0.183005) (-0.379253,-0.193239) (-0.374585,-0.203383)
(-0.369673,-0.213431) (-0.364521,-0.223379) (-0.359129,-0.233221) (-0.353499,-0.242953) (-0.347633,-0.252570)
(-0.341533,-0.262068) (-0.335201,-0.271441) (-0.328639,-0.280684) (-0.321848,-0.289794) (-0.314832,-0.298764)
(-0.307591,-0.307591) (-0.300129,-0.316270) (-0.292448,-0.324796) (-0.284550,-0.333165) (-0.276437,-0.341372)
(-0.268113,-0.349412) (-0.259580,-0.357281) (-0.250840,-0.364974) (-0.241896,-0.372487) (-0.232751,-0.379816)
(-0.223409,-0.386955) (-0.213871,-0.393902) (-0.204142,-0.400650) (-0.194223,-0.407197) (-0.184119,-0.413537)
(-0.173832,-0.419667) (-0.163366,-0.425582) (-0.152724,-0.431278) (-0.141909,-0.436751) (-0.130926,-0.441997)
(-0.119776,-0.447011) (-0.108465,-0.451790) (-0.096996,-0.456330) (-0.085372,-0.460626) (-0.073597,-0.464674)
(-0.061675,-0.468471) (-0.049611,-0.472013) (-0.037407,-0.475296) (-0.025067,-0.478315) (-0.012597,-0.481067)
(-0.000000,-0.483549) (0.012720,-0.485756) (0.025558,-0.487685) (0.038511,-0.489332) (0.051574,-0.490693)
(0.064742,-0.491764) (0.078011,-0.492542) (0.091377,-0.493024) (0.104834,-0.493204) (0.118378,-0.493081)
(0.132005,-0.492650) (0.145709,-0.491907) (0.159487,-0.490849) (0.173331,-0.489472) (0.187239,-0.487773)
(0.201204,-0.485748) (0.215221,-0.483394) (0.229285,-0.480706) (0.243390,-0.477681) (0.257532,-0.474315)
(0.271704,-0.470605) (0.285900,-0.466547) (0.300115,-0.462136) (0.314342,-0.457371) (0.328576,-0.452246)
(0.342809,-0.446757) (0.357035,-0.440901) (0.371247,-0.434674) (0.385437,-0.428071) (0.399599,-0.421090)
(0.413724,-0.413724) (0.427805,-0.405972) (0.441831,-0.397827) (0.455795,-0.389285) (0.469685,-0.380344)
(0.483492,-0.370997) (0.497204,-0.361240) (0.510808,-0.351069) (0.524292,-0.340479) (0.537640,-0.329466)
(0.550836,-0.318025) (0.563861,-0.306152) (0.576696,-0.293842) (0.589318,-0.281090) (0.601701,-0.267895)
(0.613816,-0.254251) (0.625630,-0.240157) (0.637103,-0.225610) (0.648193,-0.210611) (0.658849,-0.195160)
(0.669013,-0.179262) (0.678620,-0.162922) (0.687594,-0.146153) (0.695853,-0.128969) (0.703304,-0.111392)
(0.709849,-0.093453) (0.715387,-0.075190) (0.719821,-0.056651) (0.723062,-0.037894) (0.725037,-0.018986)
(0.725701,-0.000000)
} -- cycle;
\draw[gray,densely dashed] (0,0) circle (1);
\draw[blue!65!black,thick] plot coordinates {
(0.787910,0.000000) (0.787000,0.020608) (0.784310,0.041104) (0.779952,0.061384) (0.774086,0.081360)
(0.766895,0.100964) (0.758563,0.120145) (0.749262,0.138868) (0.739141,0.157109) (0.728328,0.174856)
(0.716929,0.192101) (0.705032,0.208840) (0.692709,0.225075) (0.680020,0.240808) (0.667013,0.256042)
(0.653730,0.270784) (0.640205,0.285037) (0.626467,0.298809) (0.612541,0.312105) (0.598449,0.324931)
(0.584209,0.337293) (0.569839,0.349198) (0.555352,0.360650) (0.540762,0.371655) (0.526081,0.382220)
(0.511320,0.392349) (0.496489,0.402049) (0.481598,0.411323) (0.466655,0.420178) (0.451669,0.428618)
(0.436647,0.436647) (0.421598,0.444272) (0.406528,0.451495) (0.391444,0.458322) (0.376353,0.464757)
(0.361261,0.470804) (0.346174,0.476468) (0.331100,0.481753) (0.316042,0.486663) (0.301008,0.491201)
(0.286003,0.495373) (0.271033,0.499181) (0.256103,0.502630) (0.241218,0.505724) (0.226384,0.508467)
(0.211606,0.510862) (0.196889,0.512913) (0.182238,0.514624) (0.167658,0.516000) (0.153155,0.517043)
(0.138733,0.517757) (0.124396,0.518147) (0.110150,0.518216) (0.096000,0.517968) (0.081949,0.517406)
(0.068003,0.516535) (0.054166,0.515359) (0.040443,0.513880) (0.026838,0.512104) (0.013356,0.510034)
(0.000000,0.507673) (-0.013225,0.505026) (-0.026314,0.502098) (-0.039264,0.498890) (-0.052070,0.495409)
(-0.064728,0.491657) (-0.077235,0.487640) (-0.089585,0.483360) (-0.101777,0.478822) (-0.113805,0.474030)
(-0.125665,0.468989) (-0.137355,0.463702) (-0.148870,0.458174) (-0.160207,0.452410) (-0.171362,0.446412)
(-0.182331,0.440187) (-0.193112,0.433737) (-0.203701,0.427068) (-0.214094,0.420184) (-0.224289,0.413089)
(-0.234282,0.405788) (-0.244070,0.398286) (-0.253650,0.390586) (-0.263019,0.382694) (-0.272174,0.374615)
(-0.281112,0.366352) (-0.289831,0.357911) (-0.298327,0.349296) (-0.306599,0.340512) (-0.314643,0.331565)
(-0.322457,0.322457) (-0.330039,0.313196) (-0.337386,0.303784) (-0.344497,0.294228) (-0.351368,0.284532)
(-0.357998,0.274701) (-0.364384,0.264741) (-0.370525,0.254655) (-0.376419,0.244449) (-0.382064,0.234129)
(-0.387458,0.223699) (-0.392600,0.213164) (-0.397487,0.202530) (-0.402119,0.191801) (-0.406494,0.180983)
(-0.410611,0.170081) (-0.414468,0.159099) (-0.418064,0.148044) (-0.421398,0.136921) (-0.424470,0.125734)
(-0.427277,0.114489) (-0.429820,0.103191) (-0.432097,0.091845) (-0.434108,0.080457) (-0.435852,0.069032)
(-0.437329,0.057575) (-0.438537,0.046092) (-0.439478,0.034588) (-0.440150,0.023067) (-0.440553,0.011536)
(-0.440688,0.000000) (-0.440553,-0.011536) (-0.440150,-0.023067) (-0.439478,-0.034588) (-0.438537,-0.046092)
(-0.437329,-0.057575) (-0.435852,-0.069032) (-0.434108,-0.080457) (-0.432097,-0.091845) (-0.429820,-0.103191)
(-0.427277,-0.114489) (-0.424470,-0.125734) (-0.421398,-0.136921) (-0.418064,-0.148044) (-0.414468,-0.159099)
(-0.410611,-0.170081) (-0.406494,-0.180983) (-0.402119,-0.191801) (-0.397487,-0.202530) (-0.392600,-0.213164)
(-0.387458,-0.223699) (-0.382064,-0.234129) (-0.376419,-0.244449) (-0.370525,-0.254655) (-0.364384,-0.264741)
(-0.357998,-0.274701) (-0.351368,-0.284532) (-0.344497,-0.294228) (-0.337386,-0.303784) (-0.330039,-0.313196)
(-0.322457,-0.322457) (-0.314643,-0.331565) (-0.306599,-0.340512) (-0.298327,-0.349296) (-0.289831,-0.357911)
(-0.281112,-0.366352) (-0.272174,-0.374615) (-0.263019,-0.382694) (-0.253650,-0.390586) (-0.244070,-0.398286)
(-0.234282,-0.405788) (-0.224289,-0.413089) (-0.214094,-0.420184) (-0.203701,-0.427068) (-0.193112,-0.433737)
(-0.182331,-0.440187) (-0.171362,-0.446412) (-0.160207,-0.452410) (-0.148870,-0.458174) (-0.137355,-0.463702)
(-0.125665,-0.468989) (-0.113805,-0.474030) (-0.101777,-0.478822) (-0.089585,-0.483360) (-0.077235,-0.487640)
(-0.064728,-0.491657) (-0.052070,-0.495409) (-0.039264,-0.498890) (-0.026314,-0.502098) (-0.013225,-0.505026)
(-0.000000,-0.507673) (0.013356,-0.510034) (0.026838,-0.512104) (0.040443,-0.513880) (0.054166,-0.515359)
(0.068003,-0.516535) (0.081949,-0.517406) (0.096000,-0.517968) (0.110150,-0.518216) (0.124396,-0.518147)
(0.138733,-0.517757) (0.153155,-0.517043) (0.167658,-0.516000) (0.182238,-0.514624) (0.196889,-0.512913)
(0.211606,-0.510862) (0.226384,-0.508467) (0.241218,-0.505724) (0.256103,-0.502630) (0.271033,-0.499181)
(0.286003,-0.495373) (0.301008,-0.491201) (0.316042,-0.486663) (0.331100,-0.481753) (0.346174,-0.476468)
(0.361261,-0.470804) (0.376353,-0.464757) (0.391444,-0.458322) (0.406528,-0.451495) (0.421598,-0.444272)
(0.436647,-0.436647) (0.451669,-0.428618) (0.466655,-0.420178) (0.481598,-0.411323) (0.496489,-0.402049)
(0.511320,-0.392349) (0.526081,-0.382220) (0.540762,-0.371655) (0.555352,-0.360650) (0.569839,-0.349198)
(0.584209,-0.337293) (0.598449,-0.324931) (0.612541,-0.312105) (0.626467,-0.298809) (0.640205,-0.285037)
(0.653730,-0.270784) (0.667013,-0.256042) (0.680020,-0.240808) (0.692709,-0.225075) (0.705032,-0.208840)
(0.716929,-0.192101) (0.728328,-0.174856) (0.739141,-0.157109) (0.749262,-0.138868) (0.758563,-0.120145)
(0.766895,-0.100964) (0.774086,-0.081360) (0.779952,-0.061384) (0.784310,-0.041104) (0.787000,-0.020608)
(0.787910,-0.000000)
} -- cycle;
\draw[blue!65!black,thick,densely dashed] plot coordinates {
(0.725701,0.000000) (0.725037,0.018986) (0.723062,0.037894) (0.719821,0.056651) (0.715387,0.075190)
(0.709849,0.093453) (0.703304,0.111392) (0.695853,0.128969) (0.687594,0.146153) (0.678620,0.162922)
(0.669013,0.179262) (0.658849,0.195160) (0.648193,0.210611) (0.637103,0.225610) (0.625630,0.240157)
(0.613816,0.254251) (0.601701,0.267895) (0.589318,0.281090) (0.576696,0.293842) (0.563861,0.306152)
(0.550836,0.318025) (0.537640,0.329466) (0.524292,0.340479) (0.510808,0.351069) (0.497204,0.361240)
(0.483492,0.370997) (0.469685,0.380344) (0.455795,0.389285) (0.441831,0.397827) (0.427805,0.405972)
(0.413724,0.413724) (0.399599,0.421090) (0.385437,0.428071) (0.371247,0.434674) (0.357035,0.440901)
(0.342809,0.446757) (0.328576,0.452246) (0.314342,0.457371) (0.300115,0.462136) (0.285900,0.466547)
(0.271704,0.470605) (0.257532,0.474315) (0.243390,0.477681) (0.229285,0.480706) (0.215221,0.483394)
(0.201204,0.485748) (0.187239,0.487773) (0.173331,0.489472) (0.159487,0.490849) (0.145709,0.491907)
(0.132005,0.492650) (0.118378,0.493081) (0.104834,0.493204) (0.091377,0.493024) (0.078011,0.492542)
(0.064742,0.491764) (0.051574,0.490693) (0.038511,0.489332) (0.025558,0.487685) (0.012720,0.485756)
(0.000000,0.483549) (-0.012597,0.481067) (-0.025067,0.478315) (-0.037407,0.475296) (-0.049611,0.472013)
(-0.061675,0.468471) (-0.073597,0.464674) (-0.085372,0.460626) (-0.096996,0.456330) (-0.108465,0.451790)
(-0.119776,0.447011) (-0.130926,0.441997) (-0.141909,0.436751) (-0.152724,0.431278) (-0.163366,0.425582)
(-0.173832,0.419667) (-0.184119,0.413537) (-0.194223,0.407197) (-0.204142,0.400650) (-0.213871,0.393902)
(-0.223409,0.386955) (-0.232751,0.379816) (-0.241896,0.372487) (-0.250840,0.364974) (-0.259580,0.357281)
(-0.268113,0.349412) (-0.276437,0.341372) (-0.284550,0.333165) (-0.292448,0.324796) (-0.300129,0.316270)
(-0.307591,0.307591) (-0.314832,0.298764) (-0.321848,0.289794) (-0.328639,0.280684) (-0.335201,0.271441)
(-0.341533,0.262068) (-0.347633,0.252570) (-0.353499,0.242953) (-0.359129,0.233221) (-0.364521,0.223379)
(-0.369673,0.213431) (-0.374585,0.203383) (-0.379253,0.193239) (-0.383678,0.183005) (-0.387858,0.172685)
(-0.391790,0.162285) (-0.395475,0.151809) (-0.398911,0.141262) (-0.402096,0.130649) (-0.405031,0.119976)
(-0.407713,0.109246) (-0.410142,0.098466) (-0.412318,0.087641) (-0.414239,0.076775) (-0.415905,0.065873)
(-0.417316,0.054941) (-0.418471,0.043983) (-0.419370,0.033005) (-0.420012,0.022012) (-0.420397,0.011008)
(-0.420525,0.000000) (-0.420397,-0.011008) (-0.420012,-0.022012) (-0.419370,-0.033005) (-0.418471,-0.043983)
(-0.417316,-0.054941) (-0.415905,-0.065873) (-0.414239,-0.076775) (-0.412318,-0.087641) (-0.410142,-0.098466)
(-0.407713,-0.109246) (-0.405031,-0.119976) (-0.402096,-0.130649) (-0.398911,-0.141262) (-0.395475,-0.151809)
(-0.391790,-0.162285) (-0.387858,-0.172685) (-0.383678,-0.183005) (-0.379253,-0.193239) (-0.374585,-0.203383)
(-0.369673,-0.213431) (-0.364521,-0.223379) (-0.359129,-0.233221) (-0.353499,-0.242953) (-0.347633,-0.252570)
(-0.341533,-0.262068) (-0.335201,-0.271441) (-0.328639,-0.280684) (-0.321848,-0.289794) (-0.314832,-0.298764)
(-0.307591,-0.307591) (-0.300129,-0.316270) (-0.292448,-0.324796) (-0.284550,-0.333165) (-0.276437,-0.341372)
(-0.268113,-0.349412) (-0.259580,-0.357281) (-0.250840,-0.364974) (-0.241896,-0.372487) (-0.232751,-0.379816)
(-0.223409,-0.386955) (-0.213871,-0.393902) (-0.204142,-0.400650) (-0.194223,-0.407197) (-0.184119,-0.413537)
(-0.173832,-0.419667) (-0.163366,-0.425582) (-0.152724,-0.431278) (-0.141909,-0.436751) (-0.130926,-0.441997)
(-0.119776,-0.447011) (-0.108465,-0.451790) (-0.096996,-0.456330) (-0.085372,-0.460626) (-0.073597,-0.464674)
(-0.061675,-0.468471) (-0.049611,-0.472013) (-0.037407,-0.475296) (-0.025067,-0.478315) (-0.012597,-0.481067)
(-0.000000,-0.483549) (0.012720,-0.485756) (0.025558,-0.487685) (0.038511,-0.489332) (0.051574,-0.490693)
(0.064742,-0.491764) (0.078011,-0.492542) (0.091377,-0.493024) (0.104834,-0.493204) (0.118378,-0.493081)
(0.132005,-0.492650) (0.145709,-0.491907) (0.159487,-0.490849) (0.173331,-0.489472) (0.187239,-0.487773)
(0.201204,-0.485748) (0.215221,-0.483394) (0.229285,-0.480706) (0.243390,-0.477681) (0.257532,-0.474315)
(0.271704,-0.470605) (0.285900,-0.466547) (0.300115,-0.462136) (0.314342,-0.457371) (0.328576,-0.452246)
(0.342809,-0.446757) (0.357035,-0.440901) (0.371247,-0.434674) (0.385437,-0.428071) (0.399599,-0.421090)
(0.413724,-0.413724) (0.427805,-0.405972) (0.441831,-0.397827) (0.455795,-0.389285) (0.469685,-0.380344)
(0.483492,-0.370997) (0.497204,-0.361240) (0.510808,-0.351069) (0.524292,-0.340479) (0.537640,-0.329466)
(0.550836,-0.318025) (0.563861,-0.306152) (0.576696,-0.293842) (0.589318,-0.281090) (0.601701,-0.267895)
(0.613816,-0.254251) (0.625630,-0.240157) (0.637103,-0.225610) (0.648193,-0.210611) (0.658849,-0.195160)
(0.669013,-0.179262) (0.678620,-0.162922) (0.687594,-0.146153) (0.695853,-0.128969) (0.703304,-0.111392)
(0.709849,-0.093453) (0.715387,-0.075190) (0.719821,-0.056651) (0.723062,-0.037894) (0.725037,-0.018986)
(0.725701,-0.000000)
} -- cycle;
\draw[->] (-1.13,0)--(1.15,0) node[right] {$\Re z$};
\draw[->] (0,-1.1)--(0,1.13) node[above] {$\Im z$};
\foreach \x in {-1,1}{\draw (\x,.018)--(\x,-.018) node[below] {$\x$};}
\foreach \y in {-1,1}{\draw (.018,\y)--(-.018,\y) node[left] {$\y$};}
\node[below left] at (0,0) {$0$};
\draw[blue!65!black] (.436647,.436647)--(.77,.78)--(1.02,.78)
 node[right] {$\mathcal D_{5/6}$};
\draw[blue!65!black] (.413724,-.413724)--(.77,-.77)--(1.02,-.77)
 node[right] {$\mathcal D_{25/32}$};
\end{tikzpicture}
\caption{The regions $\mathcal D_{25/32}\subset\mathcal D_{5/6}$ used
in the five- and six-power proofs. Their boundaries satisfy
$(1+\lambda^2/2-x)(x^2+y^2)=\lambda^2/2$.
The outer dashed circle is $|z|=1$.}
\label{fig:entropy-regions}
\end{figure}
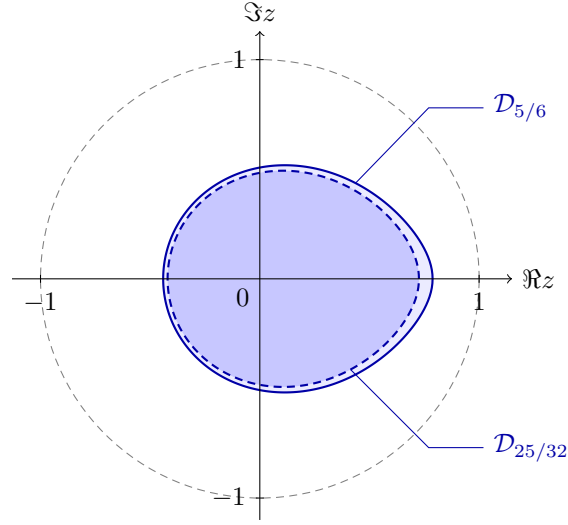

\subsection{The exceptional set}

\begin{proposition}\label{prop:large}
Fix $0<\lambda<1$ and $0<\rho<I(\lambda)$. For every positive
integer $L$, put
\begin{equation*}
 \mathcal E_\lambda(L)=
 \{\alpha\in[0,1]:z_L(\alpha)\notin\mathcal D_\lambda\}.
\end{equation*}
Then
\begin{equation}\label{eq:exceptional-measures}
 \int_{\mathcal E_\lambda(L)}\mathrm d\alpha
 \ll_{\lambda,\rho}e^{-\rho L}.
\end{equation}
\end{proposition}

\begin{proof}
Put $r=r_\lambda$ and $s=\frac12\log\frac{1+\lambda}{1-\lambda}$.
First consider $z=x+iy$ satisfying
\begin{equation*}
 u=x^2+y^2\le1,\qquad (1+r-x)u\ge r.
\end{equation*}
If $u<1$, then $u>0$ and
\begin{equation*}
 u(1-x)\ge r(1-u),\qquad
 \sqrt{\frac{2u(1-x)}{1-u}}\ge\lambda.
\end{equation*}
The choice \eqref{eq:chosen-coefficients} and
\eqref{eq:directional-value} give
\begin{equation*}
 \Phi_{\xi,\eta}(z)
 \ge s\lambda-\log\cosh s=I(\lambda)>\rho.
\end{equation*}
At $z=1$, take $(\xi,\eta)=(s,0)$. Then
\begin{equation*}
 \Phi_{s,0}(1)=s-\log\cosh s
 =I(\lambda)+(1-\lambda)s>\rho.
\end{equation*}
If $u=1$ and $x<1$, take
$t=(\rho+1)/(1-x)$ and $(\xi,\eta)=t(x-1,y)$. Direct expansion gives
\begin{equation*}
 2V_{\xi,\eta}(x,y)=t^2(1-x)(1-x^2-y^2)=0,
 \qquad \xi x+\eta y=t(1-x).
\end{equation*}
Consequently,
\begin{equation*}
 \Phi_{\xi,\eta}(z)=t(1-x)=\rho+1>\rho.
\end{equation*}

By continuity, the same pair gives $\Phi_{\xi,\eta}>\rho$ in a small
disc about the point, for $|z|\le1$. The region
$x^2+y^2\le1$, $(1+r-x)(x^2+y^2)\ge r$ is closed and bounded,
so finitely many such discs cover it. Thus there are fixed pairs
$(\xi_j,\eta_j)$, $1\le j\le m$, depending only on $\lambda,\rho$,
such that
\begin{equation*}
 \begin{gathered}
 |z|\le1,\qquad (1+r-x)|z|^2\ge r\\
 \Longrightarrow\quad
 \max_{1\le j\le m}\Phi_{\xi_j,\eta_j}(z)>\rho.
 \end{gathered}
\end{equation*}
For $\alpha\in\mathcal E_\lambda(L)$, these conditions hold with
$z=z_L(\alpha)$ by \eqref{eq:region-test}. Hence
\begin{align*}
 \int_{\mathcal E_\lambda(L)}\mathrm d\alpha
 &\le\sum_{j=1}^m
 \int_{\{\alpha\in[0,1]:\Phi_{\xi_j,\eta_j}(z_L(\alpha))>\rho\}}
 \mathrm d\alpha\\
 &\le e^{-\rho L}\sum_{j=1}^m
 \int_0^1\exp\bigl(L\Phi_{\xi_j,\eta_j}(z_L(\alpha))\bigr)
 \,\mathrm d\alpha\\
 &\le e^{-\rho L}\sum_{j=1}^m
 \exp\left(\frac{\xi_j^2+\eta_j^2}{2}\right)
 \ll_{\lambda,\rho}e^{-\rho L},
\end{align*}
by \eqref{eq:entropy-moment}.
\end{proof}

For fixed polynomials $P_1,\ldots,P_m$ and constants $c_j\ge0$,
put
\begin{equation*}
 H_\lambda=\max_{z\in\mathcal D_\lambda}
 \sum_{j=1}^m c_j|P_j(z)|^2.
\end{equation*}
The same exceptional set gives
\begin{equation}\label{eq:polynomial-region-bound}
 \alpha\notin\mathcal E_\lambda(L)
 \quad\Longrightarrow\quad
 \sum_{j=1}^m c_j|P_j(z_L(\alpha))|^2\le H_\lambda.
\end{equation}
Finally, for fixed $\sigma\in\mathbb R$ and $L$ in \eqref{eq:NL},
\begin{align*}
 \int_{\mathcal E_\lambda(L)}\mathrm d\alpha
 &\ll e^{-\rho L}=N^{-\rho/\log2+o(1)},\\
 N^{2\sigma+o(1)}\int_{\mathcal E_\lambda(L)}\mathrm d\alpha
 &\ll N^{2\sigma-\rho/\log2+o(1)}=o(N)
\end{align*}
provided
\begin{equation}\label{eq:rate-budget}
 \rho>(2\sigma-1)\log2.
\end{equation}
The next two sections choose $\lambda$ and $\rho$ with
$(2\sigma-1)\log2<\rho<I(\lambda)$.
\section{Completion of the six powers proof}\label{sec:completion}

By the definition of $I(\lambda)$ in \eqref{eq:entropy}, we have 
\begin{equation*}
 I(5/6)>\frac{81}{200}>\frac7{12}\log2.
\end{equation*}
Put $\mathcal E_6=\mathcal E_{5/6}(L)$. Proposition~\ref{prop:large}
with $\rho=81/200$ gives
\begin{equation*}
 \int_{\mathcal E_6}\mathrm d\alpha\ll e^{-81L/200}
 =N^{-\gamma_6+o(1)},\qquad
 \gamma_6=\frac{81}{200\log2}>\frac7{12}.
\end{equation*}
For fixed $a,b\ge0$ with $a+b<1$, recall
$Q(z)=z^2-az-b$ and $F(z)=L^2\cdot z\cdot Q(z/L)$, and put
\begin{equation*}
 H=\max_{z\in\mathcal D_{5/6}}|Q(z)|^2.
\end{equation*}
Then $H\ge|Q(0)|^2=b^2$ and
\begin{equation}\label{eq:goodminor}
\begin{aligned}
 \int_{\mathfrak m\setminus\mathcal E_6}
 |S(\alpha)\cdot F(G(\alpha))|^2\,\mathrm d\alpha
 &=L^4\int_{\mathfrak m\setminus\mathcal E_6}
 |S(\alpha)\cdot G(\alpha)|^2|Q(z_L(\alpha))|^2\,\mathrm d\alpha\\
 &\le HL^4\int_{\mathfrak m}|S(\alpha)\cdot G(\alpha)|^2\,\mathrm d\alpha.
\end{aligned}
\end{equation}

\begin{lemma}\label{lem:badminor}
For fixed $a,b\ge0$ with $a+b<1$ and every fixed $\kappa>0$
allowed by the major arc formula,
\begin{equation}\label{eq:badminor}
 \int_{\mathfrak m\cap\mathcal E_6}
 |S(\alpha)\cdot F(G(\alpha))|^2\,\mathrm d\alpha=o(NL^6).
\end{equation}
\end{lemma}

\begin{proof}
For $\alpha\in\mathfrak m$, Dirichlet's approximation theorem gives
\begin{gather*}
 \alpha=\frac rq+\beta,\qquad (r,q)=1,\qquad
 q\le\sqrt N,\qquad |\beta|\le\frac1{q\sqrt N},\\
 \max(q,qN|\beta|)\ge P\ge N^{11/25}>N^{5/12}.
\end{gather*}
Write
$S_{\le X}(\alpha)=\sum_{p\le X}(\log p)\cdot e(p\alpha)$.
The estimate of Maynard, Pandey and
Radziwi\l\l~\cite[Theorem~1.1]{MPR}, together with
\begin{equation*}
 \sum_{\substack{p^k\le N\\k\ge2}}\log p
 \ll N^{1/2}\log^2N,
\end{equation*}
gives
\begin{align*}
 |S_{\le N}(\alpha)|
 &\ll N^{o(1)}\left(NP^{-1/2}+N^{19/24}\right)
       +N^{1/2}\log^2N\\
 &\ll N^{19/24+o(1)}.
\end{align*}
Since $|F(G(\alpha))|\ll L^3$ and
$S(\alpha)=S_{\le N}(\alpha)-S_{\le N_1}(\alpha)$, Parseval gives
\begin{align*}
 &\int_{\mathfrak m\cap\mathcal E_6}
 |S(\alpha)\cdot F(G(\alpha))|^2\,\mathrm d\alpha\\
 &\quad\ll L^6\left(\left(\int_{\mathcal E_6}\mathrm d\beta\right)
       \sup_{\alpha\in\mathfrak m}|S_{\le N}(\alpha)|^2
       +\int_0^1|S_{\le N_1}(\alpha)|^2\,\mathrm d\alpha\right)\\
 &\quad\ll L^6\left(N^{19/12-\gamma_6+o(1)}
                  +\sum_{p\le N_1}(\log p)^2\right)\\
 &\quad\ll L^6\left(N^{19/12-\gamma_6+o(1)}+N_1\log N\right)
       =o(NL^6),
\end{align*}
where
\begin{equation*}
 \frac{19}{12}-\gamma_6<1,\qquad
 \frac{N_1\log N}{N}=N^{-\kappa}\log N=o(1).
\end{equation*}
\end{proof}

Put $v=(-b,-a,1)^T$. Lemma~\ref{lem:Jlower} gives
\begin{equation*}
 \frac{R_1(0)}{2NL}\ge\frac{11\log2}{50}+o(1),
 \qquad b^2-H\le0.
\end{equation*}
Combining \eqref{eq:majorF}, \eqref{eq:minorbase},
\eqref{eq:goodminor} and \eqref{eq:badminor}, we obtain
\begin{align*}
 \frac1{2NL^6}\int_0^1|S(\alpha)\cdot F(G(\alpha))|^2\,\mathrm d\alpha
 &\le(1-a-b)^2+v^TAv
       +(b^2-H)\frac{R_1(0)}{2NL}\\
 &\quad+H\left(\frac{C_1-2}{2}C_0R_0+\frac{\log2}{2}\right)
       +O(\varepsilon)+o(1)\\
 &\le(1-a-b)^2+v^TAv+\frac{11b^2\log2}{50}\\
 &\quad+H\left(\frac{C_1-2}{2}C_0R_0+\frac{7\log2}{25}\right)
       +O(\varepsilon)+o(1).
\end{align*}
By Lemma~\ref{lem:positive}, a sufficient condition for
$r''_6(N)\gg NL^6$ is
\begin{equation}\label{eq:six-criterion}
 v^TAv+\frac{11b^2\log2}{50}
 +H\left(\frac{C_1-2}{2}C_0R_0+\frac{7\log2}{25}\right)
 <(1-a-b)^2.
\end{equation}
Since $\mathcal D_{5/6}$ is closed and bounded, $H$ is continuous in $a,b$;
a strict inequality remains valid for sufficiently small changes
of these coefficients within $a,b\ge0$, $a+b<1$.

We take
\begin{equation}\label{eq:dyadic-choice}
 a=\frac14,\qquad b=\frac1{16},\qquad
 Q(z)=z^2-\frac z4-\frac1{16},\qquad
 v=\left(-\frac1{16},-\frac14,1\right)^T.
\end{equation}
The matrices in Section~\ref{sec:singular} give
\begin{equation*}
 v^TA_*v=\frac{54590609909}{2799360000000}<\frac1{50}.
\end{equation*}
Hence, by \eqref{eq:PSDtail},
\begin{equation}\label{eq:Gram-six}
 v^TAv<\frac1{50}
       +\left(\frac{0.254}{16}+\frac{0.077}{4}+0.04\right)^2
 <\frac{13}{500}.
\end{equation}
Direct computation gives 
\begin{equation}\label{eq:six-poly}
 |Q(z)|^2<\frac2{15}
\end{equation}
for $z\in\mathcal D_{5/6}$.

\begin{proof}[Proof of Theorem~\ref{thm:main}]
By \eqref{eq:dyadic-choice}, \eqref{eq:Gram-six} and
\eqref{eq:six-poly},
\begin{equation}\label{eq:finalmoment}
\begin{aligned}
 \frac1{2NL^6}
 \int_0^1|S(\alpha)\cdot F(G(\alpha))|^2\,\mathrm d\alpha
 <\frac{121}{256}+\frac{15}{32}+O(\varepsilon)+o(1), 
\end{aligned}
\end{equation}
where we use $C_0R_0<1.279$, $C_1=6.7814$, $\log2<0.694$. 
Lemma~\ref{lem:positive} therefore gives
\begin{align*}
 r''_6(N)
 &\ge\left(4\cdot\frac{121}{256}
       -2\left(\frac{121}{256}+\frac{15}{32}\right)
       -O(\varepsilon)+o(1)\right)NL^6\\
 &=\left(\frac1{128}-O(\varepsilon)+o(1)\right)NL^6
 >\frac1{1000}NL^6,
\end{align*}
on choosing $\varepsilon$ sufficiently small, fixing an allowed
$\kappa>0$, and then taking $N$ sufficiently large and even.
\end{proof}
\section{Five powers under GRH}\label{sec:five}

Assume GRH, use the arcs of
Lemma~\ref{lem:grhbridge}. By the definition of $I(\lambda)$ in \eqref{eq:entropy}, we have 
\begin{equation*}
 I(25/32)
 >\frac{347}{1000}>\frac12\log2.
\end{equation*}
Put $\mathcal E_5=\mathcal E_{25/32}(L)$. Proposition~\ref{prop:large}
with $\rho=347/1000$ gives
\begin{equation*}
 \int_{\mathcal E_5}\mathrm d\alpha\ll e^{-347L/1000}
 =N^{-\gamma_5+o(1)},\qquad
 \gamma_5=\frac{347}{1000\log2}>\frac12.
\end{equation*}

For fixed $a,b\ge0$, $a+b<1$, $0\le c<1$ and $\tau>0$, put
\begin{align*}
 Q(z)&=z^2-az-b,\\
 f_2(n)&=r'_2(n)-cL\cdot r'_1(n),\\
 f_3(n)&=r'_3(n)-aL\cdot r'_2(n)-bL^2\cdot r'_1(n),\\
 g_j(n)&=\max\{f_j(n),0\}\qquad(j=2,3),\\
 F_2(z)&=z^2-cLz,\qquad F_3(z)=z^3-aLz^2-bL^2z.
\end{align*}
Then $0\le g_j(n)\le r'_j(n)$ and
\begin{equation*}
 f_j(n)=\int_0^1S(\alpha)\cdot F_j(G(\alpha))e(-n\alpha)\,\mathrm d\alpha
 \qquad(j=2,3).
\end{equation*}
By \eqref{eq:firstmoment},
\begin{align*}
 \sum_{n\in U_N}g_2(n)&\ge\sum_{n\in U_N}f_2(n)
 =(1-c+o(1))NL^2,\\
 \sum_{n\in U_N}g_3(n)&\ge\sum_{n\in U_N}f_3(n)
 =(1-a-b+o(1))NL^3.
\end{align*}
Put
\begin{equation}\label{eq:jointmoment}
 \mathcal I=L^2\int_0^1|S(\alpha)\cdot F_2(G(\alpha))|^2\,\mathrm d\alpha
       +\tau^2\int_0^1|S(\alpha)\cdot F_3(G(\alpha))|^2\,\mathrm d\alpha.
\end{equation}

\begin{lemma}\label{lem:jointreflection}
With the notation above,
\begin{equation}\label{eq:jointreflection}
 r''_5(N)\ge
 \frac{2NL^6(1-c+\tau(1-a-b))^2-\mathcal I}{2\tau L}+o(NL^5).
\end{equation}
\end{lemma}

\begin{proof}
The map $n\mapsto N-n$ permutes $U_N$, and $|U_N|=N/2$.
By Cauchy's inequality and Parseval,
\begin{align*}
 2\tau Lr''_5(N)&\ge2\tau L\sum_{n\in U_N}g_2(n)g_3(N-n)\\
 &=\sum_{n\in U_N}\bigl(Lg_2(n)+\tau g_3(N-n)\bigr)^2\\
 &\quad-L^2\sum_{n\in U_N}g_2(n)^2
                  -\tau^2\sum_{n\in U_N}g_3(n)^2\\
 &\ge\frac2N\left(L\sum_{n\in U_N}g_2(n)
                   +\tau\sum_{n\in U_N}g_3(n)\right)^2-\mathcal I\\
 &\ge2NL^6\bigl(1-c+\tau(1-a-b)\bigr)^2
                 -\mathcal I+o(NL^6).
\end{align*}
\end{proof}

Write $w=(-c,1,0)^T$ and $v=(-b,-a,1)^T$, and put
\begin{equation*}
 H=\max_{z\in\mathcal D_{25/32}}
       \bigl(|z-c|^2+\tau^2|Q(z)|^2\bigr),\qquad
 D=c^2+\tau^2b^2.
\end{equation*}
Thus $H\ge D$, since $0\in\mathcal D_{25/32}$.
For a measurable set $E\subseteq[0,1]$, let $\mathcal I_E$ denote
\eqref{eq:jointmoment} with both integrals restricted to $E$. Then
\begin{equation*}
 \mathcal I_E=L^4\int_E|S(\alpha)\cdot G(\alpha)|^2
 \bigl(|z_L(\alpha)-c|^2+\tau^2|Q(z_L(\alpha))|^2\bigr)
 \,\mathrm d\alpha.
\end{equation*}
Proposition~\ref{prop:large} and \eqref{eq:grhexceptionalbridge} give
\begin{align*}
 \mathcal I_{\mathfrak m_{\mathrm G}\setminus\mathcal E_5}
 &\le HL^4\int_{\mathfrak m_{\mathrm G}}
 |S(\alpha)\cdot G(\alpha)|^2\,\mathrm d\alpha,\\
 \mathcal I_{\mathfrak m_{\mathrm G}\cap\mathcal E_5}
 &\ll L^6\int_{\mathfrak m_{\mathrm G}\cap\mathcal E_5}
       |S(\alpha)|^2\,\mathrm d\alpha=o(NL^6).
\end{align*}
By \eqref{eq:grhmixedbridge},
\begin{equation*}
 \mathcal I_{\mathfrak M_{\mathrm G}}
 =2NL^6\left((1-c)^2+\tau^2(1-a-b)^2
                 +w^TAw+\tau^2v^TAv+o(1)\right)
       +DL^5R_{1,\mathrm G}(0).
\end{equation*}
Also, \eqref{eq:grhJbridge} gives
\begin{equation*}
 \frac{R_{1,\mathrm G}(0)}{2NL}\ge\frac{\log2}{4}+o(1),
 \qquad D-H\le0.
\end{equation*}
Combining the major and minor arcs with \eqref{eq:grhminorbridge},
we obtain
\begin{align*}
 \frac{\mathcal I}{2NL^6}
 &\le(1-c)^2+\tau^2(1-a-b)^2+w^TAw+\tau^2v^TAv
       +(D-H)\frac{R_{1,\mathrm G}(0)}{2NL}\\
 &\quad+H\left(\frac{C_1-2}{2}C_0R_0+\frac{\log2}{2}\right)+o(1)\\
 &\le(1-c)^2+\tau^2(1-a-b)^2+w^TAw+\tau^2v^TAv
       +\frac{D\log2}{4}\\
 &\quad+H\left(\frac{C_1-2}{2}C_0R_0+\frac{\log2}{4}\right)+o(1).
\end{align*}
Since
\begin{equation*}
 (1-c+\tau(1-a-b))^2-(1-c)^2-\tau^2(1-a-b)^2
 =2\tau(1-c)(1-a-b),
\end{equation*}
Lemma~\ref{lem:jointreflection} gives the sufficient condition
\begin{equation}\label{eq:five-criterion}
 w^TAw+\tau^2v^TAv+\frac{D\log2}{4}
 +H\left(\frac{C_1-2}{2}C_0R_0+\frac{\log2}{4}\right)
 <2\tau(1-c)(1-a-b)
\end{equation}
for $r''_5(N)\gg NL^5$.

We now take $a=1/4$, $b=1/16$, $c=1/8$ and $\tau=2$. Then
\begin{gather*}
 Q(z)=\left(z-\frac18\right)^2-\frac5{64},\qquad
 D=\frac1{32},\qquad 1-c+\tau(1-a-b)=\frac94,\\
 w^TA_*w+4v^TA_*v
 =\frac{51822763867}{349920000000}<\frac3{20}.
\end{gather*}
By \eqref{eq:PSDtail},
\begin{equation}\label{eq:Gram-five}
\begin{aligned}
 w^TAw+4v^TAv
 &<\frac3{20}+\left(\frac{0.254}{8}+0.077\right)^2
       +4\left(\frac{0.254}{16}+\frac{0.077}{4}+0.04\right)^2\\
 &<\frac{37}{200}.
\end{aligned}
\end{equation}
Direct computation gives 
\begin{equation}\label{eq:five-poly}
 \left|z-\frac18\right|^2+4|Q(z)|^2<\frac{41}{60}
\end{equation}
for $z\in\mathcal D_{25/32}$. 

\begin{proof}[Proof of Theorem~\ref{thm:GRHfive}]
Since $(1-c)^2+\tau^2(1-a-b)^2=85/32$,
\eqref{eq:Gram-five} and \eqref{eq:five-poly} give
\begin{equation}\label{eq:fivejointbound}
\begin{aligned}
 \frac{\mathcal I}{2NL^6}
 <\frac{85}{32}+\frac{12}{5}+o(1),
\end{aligned}
\end{equation}
where the inequality follows from $C_0R_0<1.279$,
$C_1=6.7814$, $\log2<0.694$.
Lemma~\ref{lem:jointreflection} yields, for sufficiently large even $N$,
\begin{align*}
 r''_5(N)
 &\ge\left(\frac12\left(\frac{81}{16}-\frac{85}{32}-\frac{12}{5}\right)
          +o(1)\right)NL^5\\
 &=\left(\frac1{320}+o(1)\right)NL^5>\frac1{1000}NL^5.
\end{align*}
\end{proof}

\section*{Acknowledgments}
The authors acknowledge the use of OpenAI's ChatGPT during the preparation
of this manuscript. 
The authors assume responsibility for all content.
Huixi Li is supported by the National Natural Science Foundation of China (Grant No. 12561001). 
Zhixin Liu is supported by  Independent Innovation Foundation of Tianjin University (Grant No. 2025XJ21-0009).

\end{document}